\documentclass[10pt]{article}

\usepackage{amsmath,amsthm,amscd,amssymb}
\usepackage{latexsym}
\usepackage{epsf}
\usepackage{epsfig}
\usepackage{floatflt}

\newcommand{\diff}{\operatorname{Diff}}
\newcommand{\Diff}{\operatorname{Diff}}

\renewcommand{\div}{\operatorname{div}}
\newcommand{\PGL}{\operatorname{PGL}}

\newcommand{\Aut}{\operatorname{Aut}}

\newcommand{\PU}{\operatorname{PU}}
\newcommand{\interior}{\operatorname{Int}}

\newcommand{\Sph}{\mathbb{S}}
\newcommand{\R}{\mathbb{R}}
\newcommand{\C}{\mathbb{C}}
\newcommand{\Metrics}{\mathbb{M}}
\newcommand{\N}{\mathbb{N}}

\newcommand{\w}{\omega}

\newcommand{\ep}{\varepsilon}
\newcommand{\al}{\alpha}
\newcommand{\be}{\beta}

\newcommand{\la}{\lambda}
\newcommand{\na}{\nabla}

\newcommand{\SE}{\mathcal{E}}

\newcommand{\SM}{\mathcal{M}}

\newcommand{\SA}{\mathcal{A}}
\newcommand{\SH}{\mathcal{H}}

\newcommand{\SB}{\mathcal{B}}
\newcommand{\SN}{\mathcal{N}}

\newcommand{\ST}{\mathcal{T}}

\newcommand{\SQ}{\mathcal{Q}}

\newcommand{\Chi}{\mathfrak{X}}

\newcommand{\pd}{\partial}
\newcommand{\II}{\operatorname{II}}
\newcommand{\wto}{\rightharpoonup}
\newcommand{\AdmSetConst}{P}
\newcommand{\curvature}{\kappa}

\theoremstyle{plain}
\newtheorem{theorem}{Theorem}[section]
\newtheorem{lemma}[theorem]{Lemma}
\newtheorem{corollary}[theorem]{Corollary}
\newtheorem{proposition}[theorem]{Proposition}

\theoremstyle{definition}
\newtheorem{definition}[theorem]{Definition}

\theoremstyle{remark}
\newtheorem{remark}[theorem]{Remark}

\newtheorem{case[theorem]}{Case}

\title{Minimal Distortion Morphs Generated by Time-Dependent Vector Fields}
\author{Oksana Bihun${}^\ast$, Carmen Chicone${}^\ast$, and Steven G. Harris${}^\dagger$\\ 
\small{${}^\ast$ Department of Mathematics, University of Missouri, Columbia}\\
\small{${}^\dagger$ Department of Mathematics and Computer Science, Saint Louis University}}

\begin{document}

\maketitle

\begin{abstract}
A morph between two Riemannian $n$-manifolds is an isotopy between them together with the set of all intermediate manifolds equipped with Riemannian metrics. 
We propose measures of the distortion produced by some classes of morphs and 
diffeomorphisms between two isotopic Riemannian $n$-manifolds and, with respect
 to these classes, prove the existence of minimal distortion morphs and 
diffeomorphisms.  In particular, we consider the class of time-dependent 
vector fields (on an open subset $\Omega$ of $ \R^{n+1}$ in which the 
manifolds are embedded)  that generate morphs between two manifolds $M$ and 
$N$ via an evolution equation, define the bending and the morphing
distortion energies for these morphs,  and prove the existence of minimizers of the corresponding functionals in the set of time-dependent vector fields that generate morphs between $M$ and $N$ and are $L^2$ functions from
$[0,1]$ to the Sobolev space $W^{k,2}_0(\Omega,\R^{n+1})$.

\end{abstract}

\section{Introduction}
\subsection{Summary of Results}
Let $M$ and $N$ be compact and orientable smooth Riemannian $n$-manifolds isometrically embedded into $\R^{n+1}$.
A morph between $M$ and $N$ is an isotopy between them together with the set of 
all intermediate manifolds equipped with the Riemannian metrics inherited from $\R^{n+1}$. Every morph or diffeomorphism between isotopic manifolds produces distortion via stretching and bending. We define functionals that measure distortion and prove the existence of minimal distortion morphs and diffeomorphisms. 

Let $\Omega \subset \R^{n+1}$ be an open set containing the manifolds $M$ and $N$.
We define functionals $E$ and $\SE$ that measure the  distortion of  diffeomorphisms and 
morphs respectively generated by  time-dependent vector 
fields $v:\Omega\times[0,1]\to \R^{n+1}$ via the evolution equation $dq/dt=v(q,t)$ and  
prove the existence of minimizers of $E$ and $\SE$ in an admissible set $\SA^k_\AdmSetConst$ of 
time-dependent vector fields, 
which is a subset of
the closed ball of radius $\AdmSetConst$ in the Hilbert space ${\SH^k}$ of all $L^2$ functions from $[0,1]$ to the Sobolev space $W^{k,2}_0(\Omega; \R^{n+1})$, where $k \in \N$. We also analyze in detail a  concrete example of a minimal morph for the case of circles embedded in the plane.

We compute the Euler-Lagrange equations for the deformation energy functional 
defined on $\diff(M,N)$ and show that the radial map between
a manifold and its rescaled version is a critical point of the deformation energy functional.
We prove the existence of minimal deformation holomorphic 
diffeomorphisms of Riemann surfaces. 

\subsection{Background and Motivation}

A fundamental problem in Riemannian geometry and related areas is to determine whether 
two diffeomorphic compact Riemannian manifolds $(M,g_M)$  and $(N, g_N)$ are
isometric; that is, if there exists a diffeomorphism $h:M \to N$ such that
$h^\ast g_N-g_M=0$, where $h^\ast g_N$ denotes the pull-back of $g_N$ by $h$.
If no such diffeomorphism 
exists, it is important to know whether there exists a diffeomorphism 
that most closely 
resembles an isometry. This is accomplished by minimization of the \emph{deformation 
energy functional} 
\begin{equation} 
\label{eq:Funct1}
\Phi(h)= \int_M\|h^*g_N-g_M\|^2\, \omega_M, 
\end{equation} 
over the space $\diff(M,N)$ of diffeomorphisms between $M$ and $N$, where $\w_M$
is the volume form on $M$ and $\|\cdot\|$ is the fiber norm on the bundle of all
$(0,2)$-tensor fields on $M$ generated by the fiber metric $g_M^\ast\otimes
g_M^\ast$.

The minimization problem takes on added significance once the  physical interpretation 
of the tensor $ h^*g_N-g_M$ is recognized: it is exactly the (nonlinear) strain tensor 
corresponding to the deformation 
$h$ in case $g_M$ and $g_N$ are Riemannian metrics inherited from Euclidean space. 
Thus, this functional and its variants must occur in physical problems. 
Indeed, the minimal distortion problem arises, for example, in  manufacturing, computer graphics, movie making, and medical imaging. 
 
The problem of bending a sheet of metal to a desired shape using minimal 
energy  has been studied (see~\cite{Yu},~\cite{BGQ}); but,  algorithms 
for numerical approximations are proposed and used without proving the 
existence of minimizers.   

An animation might require an aesthetically pleasing transformation that takes one image to another through intermediate shapes. Such a transformation
is called a morph, or a metamorphosis (see~\cite{W} for a survey on morphing). 
A desirable morph might be defined as a
minimizer of a cost functional that measures the distortion energy. In~\cite{Terz}, 
the distortion energy of an elastically deformable surface $r:U\times[0,T]\to \R^3$, where
$U\subset \R^3$ is an open set, produced by a deformation 
$h=r(\cdot,T)$ is defined as the integral of weighted norms of the local coordinate representations
of the strain tensor $h^\ast g_T-g_0$ and the tensor $h^\ast \II_T-\II_0$, where $g_t$
and $\II_t$ are the first and the second fundamental forms of the surface $r(U,t)$ at the time $t \in
[0,T]$.

Image matching and image registration is an important subject in medical imaging. 
Image matching is used for determining the existence of abnormalities 
(distortions due to underlying medical conditions) in two images, taken at different 
times, of the same organism. The problem of registration of a population of images 
to one template for the purpose of statistical analysis is another instance of image
registration. One approach to the image matching problem is by minimization of a 
distortion functional (see~\cite{Cao, Dupuis, GrenMill, TYMeta}).

In the following sections, we will discuss the  underlying mathematical problem of 
the existence of minimal distortion diffeomorphisms and morphs between embedded manifolds of
codimension one.

\subsection{Mathematical Preliminaries, Definitions, and Results}
\label{subs:MPDR}
Given a smooth oriented $n$-manifold $S$ (perhaps with boundary) isometrically embedded into $\R^{n+1}$, we let $g_{S}$, $\w_{S}$, and $\II_{S}$ denote the Riemannian metric,  volume form, and  second
fundamental form on $S$ associated with this embedding. Also, we let  $\interior{S}$ 
(respectively, $\pd S$) denote the interior (respectively, the boundary) of the manifold $S$.

\begin{definition}
\label{df:morph}
Let $M$ and $N$ be isotopic compact connected smooth 
$n$-manifolds (perhaps with boundary) embedded in
$\R^{n+1}$ such that $M$ is oriented. A $C^\infty$ isotopy
$F:M\times [0,1] \to \R^{n+1}$ together with all the intermediate manifolds 
$M^t:=F(M,t)$, equipped with the orientations induced by the maps  
$f^t=F(\cdot,t):M \to M^t$ and  the Riemannian metrics $g_t$ inherited from $\R^{n+1}$,  
is called a (smooth) \emph{morph} from $M$ to $N$.
\end{definition}

We denote the set of all smooth (respectively, $C^r$) diffeomorphisms between 
manifolds $M$ and $N$ by $\diff(M,N)$ (respectively, $\diff^r(M,N)$). Similarly, 
we denote the set of all smooth morphs between $M$ and $N$ by $\SM(M,N)$. 
If $F$ is an 
isotopy, then each map $F(\cdot,t):M \to M^{t}$ induces
smooth diffeomorphisms $\interior F(\cdot,t):\interior M \to\interior  M^{t}$ and  $\pd F(\cdot,t):\pd M \to\pd  M^{t}$ by
restriction.

In addition, we consider morphs between manifolds $M$ and $N$ with different regularity properties. 
For example, we let  $\SM^{r,\text{ac}}(M,N)$ denote the set of all continuous isotopies $F:M\times[0,1]\to \R^{n+1}$ between  $M$ and $N$ such that for each $p \in M$ the map
$t\mapsto F(p,t)$ is absolutely continuous on $[0,1]$ and for each $t \in [0,1]$ the function 
$p \mapsto F(p,t)$ 
is a $C^r$ diffeomorphism from $M$ onto its image. As in the case of smooth
morphs, the diffeomorphism $F(\cdot,t):M\to M^t$ induces an orientation on the
intermediate manifold $M^t$.

There are several choices for 
cost functionals that measure the distortion of a diffeomorphism 
$h \in \diff(M,N)$ or a morph $F \in \SM(M,N)$. 

A complete theory of the existence of minimizers of cost functionals that 
measure distortion of diffeomorphisms and morphs 
due to change of volume is presented in~\cite{RMJpaper}. 
In this case, the value of the distortion energy functional at a diffeomorphism $h:M \to N$  is defined to be 
the square of the infinitesimal
relative change of volume $|J(h)|-1$ produced by $h$ integrated over 
the manifold $M$, where $J(h)$ is the Jacobian determinant of $h$. This functional
does not take into account the distortion of shape produced by $h$, which is captured
by functionals~\eqref{eq:Funct1} and~\eqref{eq:Funct2},
where the fiber metric $\|\cdot\|$ on the bundle of all tensor fields of type $(0,2)$ is induced by 
the fiber inner product $g_{M}^{\ast}\otimes g_{M}^{\ast}$ (see~\cite{KN}).

The general problem of the existence of minimizers of $\Phi$ is open. The special case where $M$ and 
$N$ are one-dimensional is studied in~\cite{OpusculaPaper} where, among other results, the functional 
$\Phi$ is shown to have no minimizer  in case $M$ and $N$ are circles with the radius of $N$ smaller 
than the radius of $M$. Thus, a solution of the general problem must take into account at least some global properties of the metric structures of the manifolds $M$ and $N$. 
On the other hand, we will prove the existence of  minimizers in case
$M$ and $N$ are Riemann spheres or compact Riemann surfaces of genus greater
than one and the admissible set is  $\operatorname{HD}(M,N)=\{h \in \diff(M,N): h \mbox{ is a holomorphic map}\}$. More precisely, the following theorem will be proved in section~\ref{Ch2:2dimSphere}.
\begin{theorem}
\label{RiemSph}
(i) Let $h_R: \mathbb{R}^3\to  \mathbb{R}^3$ be the radial map given by 
$h_R(p)=R p$ for some number $R>0$. If  $M=\mathbb{S}^2$ is the 
$2$-dimensional unit sphere isometrically embedded into $\R^3$ and  $N=h_R(M)$,  
then  $h\in \operatorname{HD}(M,N)$ is a global minimizer 
of the functional $\Phi$, restricted to the admissible set 
$\operatorname{HD}(M,N)$,  if and only if $h=f \circ h_R|_M$, where $f$ is 
an isometry of $N$.\\
(ii) Let $M$  and $N$ be compact Riemann surfaces. If
 $\operatorname{HD}(M,N)$ is not empty and
the genus of $M$ is at least two, then
there exists a minimizer of the functional $\Phi$ in $\operatorname{HD}(M,N)$. 
\end{theorem}

We note that the diffeomorphism $h=f \circ h_R|_M$ in the latter theorem satisfies the Euler-Lagrange 
equations for the functional $\Phi$ with its natural domain $\diff(M,N)$ 
(see appendix~\ref{App:PhiEL}). The Euler-Lagrange equations for the 
deformation energy functional
$\Phi:\diff(M,N)\to \R_+$ are highly nonlinear and rather complicated (see
proposition~\ref{fundlemma1}),
which discouraged us from using them to show the
existence of minimizers of the functional $\Phi$ on $\diff(M,N)$.

If we wish to match, in addition to the Riemannian metrics, the embeddings of the manifolds $M$ and $N$ (to avoid, for example, zero distortion energy maps
between a square and a round cylinder in $\R^3$), we arrive at the problem of minimization of the functional
\begin{equation} 
\label{eq:Funct2} 
\Lambda(h):= \int_M\|h^*g_N-g_M\|^2\, \omega_M+\int_M\|h^*\II_N-\II_M\|^2\, \omega_M
\end{equation}  
over the space of diffeomorphisms between $M$ and $N$, where $\II_M$ and $\II_N$ are the
second fundamental forms on the manifolds $M$ and $N$.

One of the difficulties encountered in attempts to minimize $\Phi$ over 
$\diff(M,N)$ is the lack of a complete understanding of the structure of this 
infinite-dimensional space. The natural new approach is to linearize; that is,
replace $\diff(M,N)$ with a subset of a linear function space. Using this 
approach, which already appears in the literature on 
image deformation (see ~\cite{Cao, Dupuis, GrenMill, TYMeta}), we define our distortion 
energy functionals on time-dependent vector fields that generate morphs 
(see Fig.~\ref{ppict1}).

\begin{figure}
\begin{center} 
\includegraphics[width=3in]{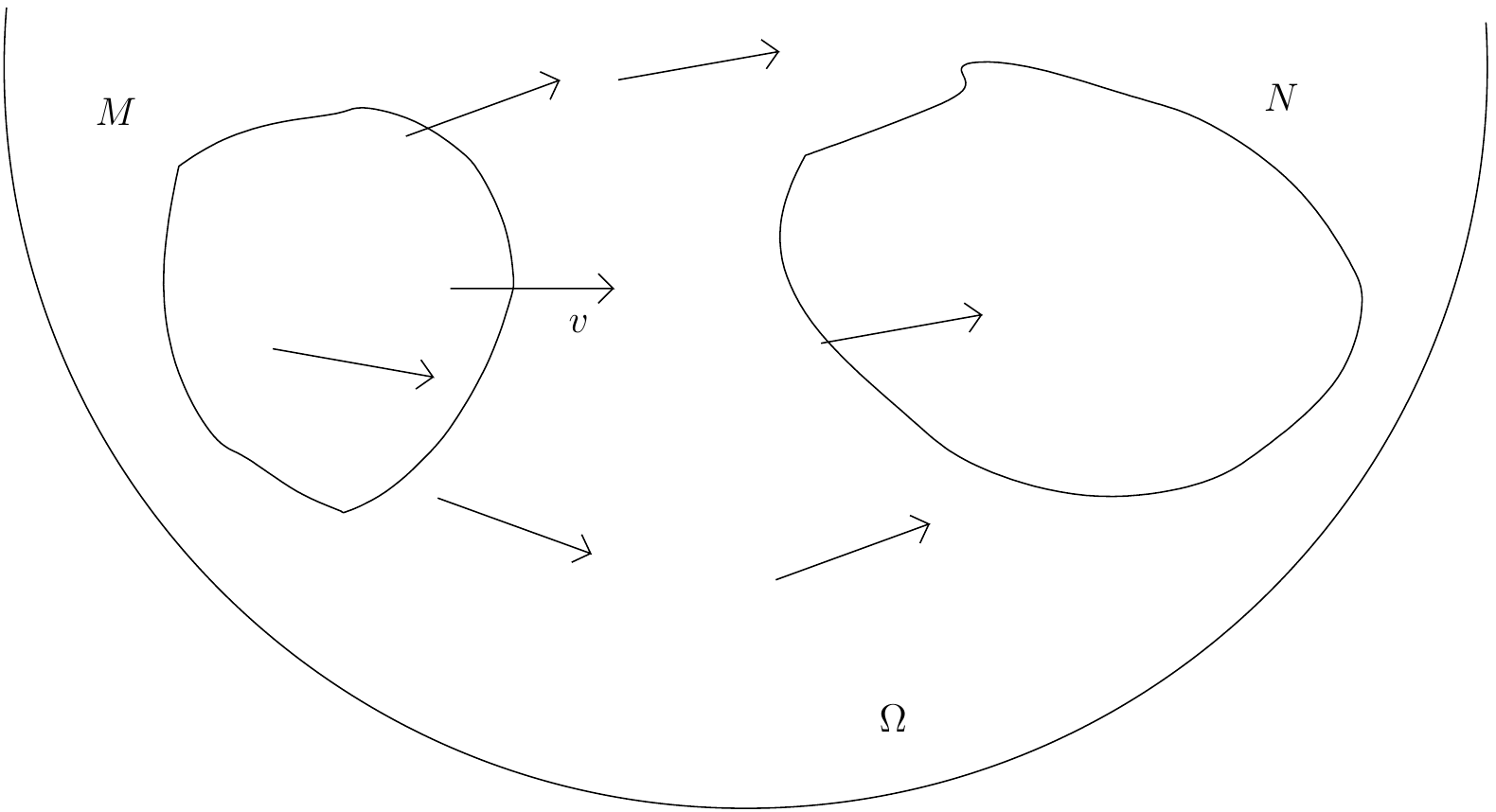}
\caption{The time-dependent vector field $v:\Omega\times[0,1]\to\R^{n+1}$ generates the morph $F^v(p,t)$, which is the solution of
the initial value problem $dq/dt=v(q,t), \; q(0)=p$.}
\label{ppict1}
\end{center}
\end{figure}

Let us denote the Euclidean norm of an element $A \in \R^m$ by $|A|$ or by 
$|A|_{\R^m }$.
Let $\Omega \subset \R^{n+1}$ be an open ball containing the manifolds 
$M$ and $N$, $C_c^\infty(\Omega;\R^{n+1})$ the space of all smooth functions from $\Omega$ to $\R^{n+1}$ with compact
support, and 
$V^k:=W^{k,2}_0(\Omega;\R^{n+1})$ the closure of $C_c^\infty(\Omega;\R^{n+1})$ 
in the Sobolev space $W^{k,2}(\Omega;\R^{n+1})$ (see~\cite{Evans}).

The space $V^k$ is a Hilbert space with the inner product
$$
\langle f,g \rangle_{V^k}=
\sum_{i=1}^{n+1}\sum_{\al, |\al|\leq k}\int_{\Omega}D^\al f_i D^\al g_i dx,
$$
where $f=(f_1,\ldots,f_{n+1}):\Omega\to \R^{n+1}$, $\al=(\al_1,\ldots,\al_{n+1})$ is a multi-index with nonnegative integer components, 
$|\al|=\al_1+\ldots+\al_{n+1}$, and $D^\al f_i={\pd^{|\al|}f_i}\big /{\pd x_1^{\al_1}\ldots\pd x_{n+1}^{\al_{n+1}}}$ is 
the corresponding weak partial derivative of $f_i$.
We choose $k \in \N$ 
large enough so that the Sobolev space $W^{k,2}_0(\Omega)$ is embedded into 
$ C^r(\bar{\Omega})$ and $r  \geq 1$.  By the Sobolev Embedding Theorem (see~\cite{Adams, Evans}), it suffices to choose
$k\geq ({n+1})/{2}+r+1$. 

 Consider time-dependent vector fields 
$v:\Omega\times[0,1]\to\R^{n+1}$ on $\Omega$ that belong to the Hilbert 
space 
\begin{equation}
\label{df:SH}
{\SH^k}=L^2(0,1;V^k)
\end{equation}
(see Fig.~\ref{ppict1}).
By an abuse of notation,  we will sometimes write $v(x,t)=v(t)(x)$ for $v \in {\SH^k}$ and 
$(x,t) \in \Omega \times[0,1]$. A time-dependent vector field $v:\Omega\times[0,1]\to\R^{n+1}$ belongs to the Hilbert space ${\SH^k}$ if its norm $\|v\|_{\SH^k}:=(\int_0^1\|v(\cdot,t)\|_V^2\,dt)^{\frac{1}{2}}$ is finite.
The inner product on ${\SH^k}$ is defined by
$$
\langle v,w \rangle_{\SH^k}=\int_0^1\langle v(\cdot,t),w(\cdot,t) \rangle_{V^k}\,dt.
$$ 

Every vector field 
$v \in {\SH^k}$ generates a morph $F^v:M\times[0,1]\to\R^{n+1}$ from $M$ to 
$F^v(M,1)$ via the evolution equation
\begin{equation}
\label{eq:MorphGen}
 \frac{dq}{dt}=v(q,t).
\end{equation}
More precisely, let $\eta^v(t;t_0,x)$ be the evolution operator of equation~\eqref{eq:MorphGen}; that is,  for every $t_0\in[0 ,1]$ and $x \in \Omega$ the function 
$t \mapsto \eta^v(t;t_0,x)$ solves equation~\eqref{eq:MorphGen} and satisfies the initial condition $\eta^v(t_0;t_0,x)=x$. The morph $F^v$ is defined by 
$F^v(p,t)=\eta^v(t;0,p)$ for all $(p,t)\in M \times [0,1]$. By the properties of the evolution operator $\eta^v$, which have been studied 
in~\cite{Dupuis} and~\cite{TrouveYounes}, the morph $F^v(p,t)$ is of class
$\SM^{r,\text{ac}}(M,F^v(M,1))$ (see
lemmas~\ref{Lemma:exist} and~\ref{lemma:Diff}). The time-one map of 
the evolution operator $\eta^v$ is defined to be $\phi^v(x):=\eta^v(1;0,x)$ for all $x \in \Omega$, and we define $\psi^v=\phi^v|_M$. 

Let $\SA^k_\AdmSetConst$ be the admissible set of all time-dependent vector fields in ${\SH^k}$ that generate morphs between the manifolds
$M$ and $N$ and are bounded 
by a uniform positive constant $\AdmSetConst$. In symbols, 
\begin{equation}
\label{df:SA}
\SA^k_\AdmSetConst=\{v \in {\SH^k}: \psi^v \in\diff^r(M,N) \mbox{ and } \|v\|_{{\SH^k}}\leq \AdmSetConst\}.
\end{equation}
We will prove that for  $\AdmSetConst$ sufficiently large, the admissible set $\SA^k_\AdmSetConst$ is nonempty and $\SA^k_\AdmSetConst$ is weakly closed in 
${\SH^k}$ (see lemma~\ref{lemma:PullBackConvergence}).

Let $\ST^r{}_s(M)$ denote the set of all continuous tensor fields on $M$ 
contravariant of order $r$ and covariant of order $s$ (also called type $(r,s)$).
For  a tensor field $\tau_N \in \ST^0{}_s(N)$ and a diffeomorphism $h:M\to N$, 
$h^\ast \tau_N$ denotes the pull-back of $\tau_N$ to $M$. 

For each $t \in [0,1]$ and $v \in \SA^k_\AdmSetConst$, the manifold $M^{v,t}:=F^v(M,t)$  is  called an intermediate state of the morph $F^v$  between manifolds $M$ and $N$ generated by the time-dependent vector field $v$. We endow this intermediate state with the  Riemannian metric $g^v_t$ inherited from its embedding in $\R^{n+1}$ and let  $\II^v_t$ denote the corresponding second fundamental form.

\begin{definition}
\label{df:ESE}
Let $B_1$ and $B_2$ be nonnegative real numbers, $F^v$ the morph, and $\phi^v$ the time-one map generated by the time-dependent vector field $v \in \SA^k_\AdmSetConst\subset {\SH^k}$ via the evolution equation~\eqref{eq:MorphGen}. Recall that $\psi^v:=\phi^v|_M$.
The \emph{bending distortion energy} of  $v$ is 
\begin{eqnarray*}
 E(v)=E(v;B_1,B_2)&=& B_1 \int_M\|(\psi^v)^\ast g_N-g_M\|^2\w_M\\
&&{}+B_2 \int_M\|(\psi^v)^\ast \II_N-\II_M\|^2\w_M
\end{eqnarray*}
and the \emph{morphing distortion energy} of $v$ is
\begin{eqnarray*}
\SE(v)=\SE(v;B_1,B_2)&=&B_1 \int_0^1\int_M\|F^v(\cdot,t)^\ast g^v_t-g_M\|^2\,\w_M dt\\
&&{}+B_2 \int_0^1\int_M\|F^v(\cdot,t)^\ast \II^v_t-\II_M\|^2\,\w_M dt,
\end{eqnarray*}
where  $\|\cdot\|$ is the fiber norm on the tensor bundle $\ST^0{}_2(M)$ generated by the fiber inner product
$g_M^\ast \otimes g_M^\ast$. (Note: We will use the same notation for the fiber norm on the tensor bundle $\ST^0{}_s(M)$
generated by the inner product $\otimes_{i=1}^{s} g_M^\ast$.)
\end{definition}

We will prove that the functionals $E$ and $\SE$ have minimizers in $\SA^k_P$.

\begin{theorem}\label{MainTheorem}
(i) If $\AdmSetConst>0$ and $k \in \N$ are sufficiently large, then 
each of the functionals
  $E:\SA^k_\AdmSetConst \to\R_+$ and $\SE:\SA^k_\AdmSetConst \to\R_+$
has a minimizer in the admissible set
$\SA^k_\AdmSetConst$.
\end{theorem}
The detailed conditions on the constants $\AdmSetConst$
 and $k$ are formulated in theorem~\ref{th123}.

We note that each diffeomorphism $\psi^v:M \to N$ generated by a time-dependent vector field $v \in \SA^k_P$ is isotopic, as a map
from $M$ to $\R^{n+1}$, to the inclusion map $i:M \to \R^{n+1}$ via the isotopy 
$F^v\in \SM^{r,\text{ac}}(M,N)$. To minimize the distortion energy of
diffeomorphisms from other isotopy classes, we replace the map $\psi^v$ in the definition of the functional $E$ by the diffeomorphism
$\psi^v\circ \phi:M\to N$, where $\phi$ is a fixed diffeomorphism on $M$. 
The existence of minimizers of the functional $E$ with the above adjustment guarantees the existence of minimizers of the 
functionals $\Phi$ and $\Lambda$ defined in displays~\eqref{eq:Funct1} and~\eqref{eq:Funct2} in a restricted
admissible set of all $C^2$ diffeomorphisms between the manifolds $M$ and $N$, which, considered as maps from $M$ to $\R^{n+1}$, are isotopic to
a given map $\phi:M \to \R^{n+1}$.

\begin{theorem}
\label{MainThm2}
If $P>0$ and $k \in \N$ are sufficiently large, then for every $\phi \in \diff(M)$ both functionals
$\Phi$ and $\Lambda$ defined in displays~\eqref{eq:Funct1} and~\eqref{eq:Funct2} respectively have minimizers in the admissible set
$$\SB^k_{P,\,\phi}:=\{h \in \diff^2(M,N): h=\psi^v\circ \phi \text{ for some } v\in \SA^k_P\}.$$
\end{theorem}

In section~\ref{Sec:Example} we construct an example of a minimal distortion
diffeomorphism and morph
between the unit circle $\Sph^1$ and the circle $\Sph^1_R$, with radius $R>1$,
in
$\R^2$.

While the construction of a minimizer of the functional $E$
does not cause significant difficulties, 
finding a minimizer of the functional $\SE$ is a much more intricate process.
Even after we restrict our attention to the family of 
morphs whose intermediate states are concentric 
circles, finding a minimal distortion morph requires delicate analysis, which is done in section~\ref{Sec:Example}.

To find a morph $H(p,t)=\psi(t)p$ with $\psi \in Q_+:=
\{\phi \in C^2(0,1)\cap C[0,1]: \phi(0)=1,\, \phi(1)=R, 
\mbox{ and } \phi \mbox{ is increasing}\}$,  which has minimal
distortion among the morphs $F \in \SM^{3,\text{ac}}(M,N)$ whose intermediate states
are circles with increasing radii, we solve the optimization problem 
\begin{equation}
\label{Sec1OptPr1}
\begin{array}{l}
\mbox{minimize } J(\psi)=\int_0^1 (\psi^2-1)^2\,dt+\int_0^1 (\psi-1)^2\,dt\\
\mbox{for } \psi \in Q_+\\
\mbox{subject to }
\int_0^1\Big(\frac{\psi'}{\psi}\Big)^2\,dt\leq A,
\end{array}
\end{equation}
where $A>0$. The inequality constraint in optimization problem~\eqref{Sec1OptPr1}
is derived from the requirement that the vector fields on the set $\Omega\subset
\R^2$ generated by the morph $H$ must be bounded by a uniform constant.

More precisely, let $\Omega$ be the open ball in $\R^2$ of radius $R+2$ and let
$\rho:\R^2\to\R^2$ be a bump function such that $\rho \equiv 1$ on the open ball $B(0,R+1)$,
$\rho\equiv 0$ on $\Omega^c$, and $0\leq\rho\leq 1$.

Given $P>0$, define
$$
A(P):=\|\rho \operatorname{id}_{\Omega}\|^{-2}_{W_0^{5,2}(\Omega;\R^2)} P^2.
$$

\begin{theorem}
If the constant $A=A(P)>\log^2 R$,
then there exists a unique minimal distortion morph $H(p,t)=\psi(t) p$, where $\psi \in Q_+$, between the unit circle
$\Sph^1$ and the circle of radius $R>1$ in $\R^2$, 
among all the morphs $F \in \SM^{3,\,\operatorname{ac}}(M,N)$ of the form
$$
F(p,t)=\phi(t) p,  \;\; \phi \in Q_+
$$
that generate the time-dependent vector field
$$
v(x,t)=\frac{\phi'(t)}{\phi(t)}\rho(x) x, \;\; (x,t)\in\Omega\times[0,1]
$$
such that $\|v\|_{\SH^5}\leq P$.

Moreover, the radial function 
 $\psi \in Q_+$ of the distortion minimal morph $H$ is the unique solution of the 
optimization problem~\eqref{Sec1OptPr1} and 
solves the initial value problem
\begin{equation}
\label{eq:psiIVP}
\left\{
\begin{array}{ll}
\psi'=\frac{1}{\sqrt{\la}}\psi \sqrt{\mu+(\psi^2-1)^2+(\psi-1)^2},\\
\psi(0)=1,
\end{array}
\right.
\end{equation}
where the pair of positive constants $\la$ and $\mu$ is the unique solution 
 of the system of equations
\begin{equation}
\label{eqHyp:CondLaE1}
\int_1^R \frac{ds}{s \sqrt{\mu+(s^2-1)^2+(s-1)^2}}=\frac{1}{\sqrt{\la}}
\end{equation}
and
\begin{equation}
\label{eqHyp:CondLaE2}
\frac{1}{\sqrt{\la}}\int_1^R \frac{\sqrt{\mu+(s^2-1)^2+(s-1)^2}}{s}\,ds=A.
\end{equation}
\end{theorem}

\section{Bending and Morphing via Time-Dependent\\ Vector Fields in $\R^{n+1}$}
\label{Label123}
In this section we prove theorem~\ref{MainTheorem}.

We will show that the admissible set $\SA^k_\AdmSetConst$ is nonempty if $\AdmSetConst$ is sufficiently large.

\begin{lemma}
Let $M$ and $N$ be manifolds as in definition~\ref{df:morph}. 
Let $F$  be a smooth morph between the manifolds $M$ and $N$ and assume that $\Omega \subset \R^{n+1}$
is an open ball in $\R^{n+1}$ containing the image $F(M\times[0,1])$ of the morph $F$.
There exists $\AdmSetConst_{0}>0$ such that the admissible set $\SA^k_\AdmSetConst$ is nonempty 
whenever $\AdmSetConst\geq\AdmSetConst_{0}$ and $k\geq\frac{n+1}{2}+2$.
\end{lemma}
\begin{proof}
The morph $F\in \SM(M,N)$ defines the $\R^{n+1}$ valued function 
$$
v(y,t)=\frac{\pd F}{\pd t}([F(\cdot,t)]^{-1}(y),t)
$$
on the compact subset $Q=\{(F(x,t),t):(x,t)\in M\times[0,1]\}$ of $\R^{n+1}\times\R$. 

We will extend the function $v$ to a smooth vector field $w\in \SH^k$ such that $\psi^w(M)=N$.

First, notice that the smooth map $G:M\times[0,1]\to\R^{n+1}\times\R$ defined by
$G(x,t)=(F(x,t),t)$ is a proper map ($M \times[0,1]$ is compact) and an injective immersion, 
hence an embedding (see~\cite{AMR}). Therefore, 
$Q=G(M\times[0,1])$ is a submanifold (with boundary) of $\R^{n+1}\times\R$ (see~\cite{Hirsch}).

Next, notice that the map $G_1:Q\to M\times[0,1]$ defined by $G_1(y,t)=([F(\cdot,t)]^{-1}(y),t)$
is the inverse of $G$. Because $G$ is an immersion, hence a local diffeomorphism, the map $G_1$ is smooth. Therefore, the map
$v:Q\to \R^{n+1}$ is smooth because it is the composition of two smooth maps $G_1:Q\to M\times[0,1]$ and $\frac{\pd F}{\pd t}:M\times[0,1]\to \R^{n+1}$.

The smooth function $v:Q\to \R^{n+1}$ can be extended locally. That is,
 for every $(y,t)\in Q$ there exists an open set
$U\subset \R^{n+1}\times\R$ such that $(y,t)\in U$ and a smooth function $v_1:U\to \R^{n+1}$
such that $v_1|_{U\cap Q}=v|_{U\cap Q}$. This local extension property 
follows from a more general fact about smooth functions defined on submanifolds:
Let $S$ be an $s$-dimensional smooth submanifold (perhaps with boundary) of $\R^{m}$ and let $f:S\to\R$ be a smooth function. Then for every $x\in S$ there exists an open set
$W\subset \R^m$ containing $x$ and a smooth function $f_1:W\to\R$ such that $f|_{W\cap S}=f_1|_{W\cap S}$.
It is easy to construct 
a local extension of the function $f$ using submanifold charts on $S$ and
the definition of a smooth function whose domain is a submanifold with boundary. The details
are left to the reader.

 Therefore, the function $v:Q\to \R^{n+1}$ satisfies the conditions of 
the smooth Tietze extension theorem (see \cite{AMR})
and can be extended from the closed set $Q\subset \R^{n+1}\times \R$ by
a smooth map $\bar{v}:\R^{n+1}\times \R\to \R^{n+1}$.

Finally, define $w(x,t)=\rho(x)\bar{v}(x,t)$, where $\rho:\R^{n+1}\to \R$ is a smooth bump function such that $\rho\equiv 1$ on $Q$ and $\rho\equiv 0$ on $\pd \Omega$, 
and set
$P_0:=\|w\|_{\SH^k}$.
\end{proof}

From now on, we assume that the open set $\Omega$ in $\R^{n+1}$ is chosen as in the latter lemma and the constant $P>0$ is
large enough so that the set $\SA^k_P$ is not empty; the number $k$ of weak derivatives satisfies the inequality $k\geq (n+1)/2+r+1$, where $r\geq 1$.

For each $v \in \SA^k_\AdmSetConst$,  the time-one map $\psi^{v}:M \to N$ transforms  the interior (respectively, the boundary)
of the manifold $M$ to the interior (respectively, the boundary)  of the manifold $N$.
The existence and the convergence properties of the evolution operators generated by vector fields $v \in \SH^k$ via the evolution equation~\eqref{eq:MorphGen}
have been studied in~\cite{Dupuis, TrouveYounes}. For convenience of the reader, we state some of these properties 
(which will be useful in our proofs) in Appendix~\ref{Append:A}.

As mentioned in section~\ref{subs:MPDR}, every time-dependent vector 
field $v \in \SA^k_\AdmSetConst \subset \SH^k$ generates a morph $F^v$
between the manifolds $M$ and $N$ of class $\SM^{r,\text{ac}}(M,N)$. 

Let us recall the distortion energy functionals $E:\SH^k \to \R_+$ and $\SE:\SH^k\to \R_+$ (see definition~\ref{df:ESE}). 

One of the main ingredients in the proof of theorem~\ref{MainTheorem} is the 
weak continuity of the functionals $E$ and $\SE$. We will prove the weak continuity
of more general auxiliary functionals, where the tensor fields
$\tau_M$ and $\tau_N$ in the following definition will later be replaced by the
first and the second fundamental forms on the manifolds $M$ and $N$ respectively.

\begin{definition}
\label{df:I12}
Let $M$ and $N$ be manifolds as in definition~\ref{df:morph}.
For given continuous tensor fields  $\tau_M$ and $\tau_N$ of type $(0,s)$ on $M$ and $N$ respectively, the functional
$J:\SA^k_\AdmSetConst \to \R_{+}$ is defined by
$$
J(v)=\int_M\|(\psi^v)^\ast \tau_N-\tau_M\|^2\, \w_M.
$$
Let  $v \in \SA^k_\AdmSetConst$ be  a time-dependent vector field  that generates a morph 
$F^v\in\SM^{r,\text{ac}}(M,N)$ from the manifold $M$ to $N$.
Recall that the intermediate state at the time $t\in[0,1]$ of the morph $F^v$ is denoted by $M^{v,t}$. The Riemannian metric and the second fundamental form
on $M^{v,t}$, which are associated with the embedding of $M^{v,t}$ into $\R^{n+1}$,  are denoted by $g^v_t$ and $\II^v_t$ respectively.
The functionals $I_{1}:\SA^k_\AdmSetConst \to \R_{+}$ and $I_{2}:\SA^k_\AdmSetConst \to \R_{+}$ are given by
\begin{equation}
\label{df:I1}
I_1(v)=\int_0^1 \int_M\|F^v(\cdot,t)^\ast g_t^v-g_M\|^2 \,\w_M\,dt
\end{equation}
and
\begin{equation}
\label{df:I2}
I_2(v)=\int_0^1 \int_M\|F^v(\cdot,t)^\ast \II_t^v-\II_M\|^2 \,\w_M\,dt.
\end{equation}
\end{definition}

\begin{definition}
\label{df:MultilinearNorm}
Let $X$ be an $n$-dimensional vector space equipped with the inner product $g_{X}$. Let
$T^{0}{}_{s}(X)=\otimes_{i=1}^{s} X^{\ast}$. For every $v \in X$, we denote the norm of $v$ with respect to the inner
product $g_{X}$ by $|v|_{g_{X}}=g_{X}(v,v)^{1/2}$ and the unit sphere by $S_{g_{X}}=\{v \in X:|v|_{g_{X}}=1\}$.
Define the norm on $T^{0}{}_{s}(X)$ by
$$
\|b\|_{g_{X}}=\max_{v_{i}\in S_{g_{X}}}|b(v_{1},\ldots,v_{s})|.
$$
Another norm on $T^0{}_s(X)$ is defined by (see \cite{KN})
$$
\|b\|=\otimes_{i=1}^s g_X^\ast(b).
$$

Let $\Metrics (X)\subset T^{0}{}_{2}(X)$ denote the metric space of all inner
products on $X$ with the metric $d(g,g')=\|g-g'\|_{g_{X}}$. 
\end{definition}

Note that if $\{e_1,\ldots,e_n\}$ is an orthonormal 
basis of $(X,g_X)$, then 
\begin{equation}
\label{eq:gXastnorm}
\|b\|=\sum_{i_1,\ldots,i_s=1}^n
b(e_{i_1},\ldots,e_{i_s})^2.
\end{equation}

\begin{lemma}
\label{lemma:NormsCont}
The function $\eta:T^{0}{}_{s}(X)\times \Metrics (X)\to \R$ defined by $\eta(\beta,g)=\|\beta\|_{g}$ is continuous on its
domain.
\end{lemma}
The proof of the latter lemma is sketched in Appendix~\ref{Append:B} for
completeness.

For a $C^{1}$ Riemannian manifold $(S,g_{S})$, let $\|\cdot\|$ be the fiber norm on the bundle of all continuous tensor
fields on $S$ of type $(0,s)$ generated by the fiber inner product 
$\otimes_{i=1}^{s}g_{S}^{\ast}$. 

\begin{lemma}
\label{lemma:MUltNormBded}
If $b$ is a continuous tensor field of type $(0,s)$ on a $C^1$  Riemannian $n$-manifold $(S, g_S)$,  then the function
\begin{equation}
\label{df:TensorNormCont}
 z \mapsto \|b(z)\|_{g_{S}(z)}
\end{equation}
is continuous on $S$.
Moreover, the norms $\|\cdot\|$ and $\|\cdot\|_{g_{S(z)}}$ are uniformly equivalent; in fact,
\begin{equation}
\label{eq:NormsEquiv}
\frac{1}{n^{s/2}}\|b(z)\|\leq\|b(z)\|_{g_{S}(z)}\leq \|b(z)\| 
\end{equation}
for all $b \in \ST^{0}{}_{s}(S)$ and  $z \in S$.
\end{lemma}
\begin{proof}The continuity of the function defined in display~\eqref{df:TensorNormCont} follows immediately from lemma~\ref{lemma:NormsCont}; and, the
inequalities in display~\eqref{eq:NormsEquiv} can be easily derived from the definitions of the norms $\|\cdot\|$ and
$\|\cdot\|_{g_{S}(z)}$ and formula~\eqref{eq:gXastnorm}.
\end{proof}

Recall that a sequence $\{v^{l}\}_{l=1}^{\infty}\subset \SH^k$ converges weakly to $v \in \SH^k$ (in symbols, $v^{l}\wto v$) as $l \to \infty$ if 
$\langle v^{l}-v,w\rangle \to 0$ as $l\to \infty$ for every $w \in \SH^k$. We call $v$ the weak limit of
$\{v^{l}\}_{l=1}^{\infty}$.
A set $\SQ \subset \SH^k$ is
sequentially weakly closed if it contains the weak limit of every weakly convergent sequence
$\{v^{l}\}_{l=1}^{\infty}\subset \SQ$.

\begin{lemma}
\label{lemma:PullBackConvergence}
(i) The admissible set $\SA^k_\AdmSetConst$ is  sequentially weakly closed in $\SH^k$.\\
(ii) Let $b$ be a continuous tensor field of type $(0,s)$ on the manifold $N$; and, 
 for every $w \in \SA^k_\AdmSetConst$, let $\psi^w$ denote the restriction to $M$ of the time-one map of the evolution equation ${dq}/{dt}=w(q,t)$. 
If a sequence $\{v^l\}_{l=1}^\infty \subset \SA^k_\AdmSetConst$ converges weakly to $v \in \SA^k_\AdmSetConst$ in $\SH^k$,  then 
$$
\lim_{l\to\infty}\|(\psi^{v^l})^\ast b- (\psi^v)^\ast b\|(p_0)=0
$$
for every $p_0 \in M$.
\end{lemma}
\begin{proof}
(i)
 Let $\{v^l\}_{l=1}^\infty \subset \SA^k_\AdmSetConst$ and suppose that $v^l$ converges weakly
to some $v \in \SH^k$ as $l \to \infty$.
 We
will show that $v \in \SA^k_\AdmSetConst$.

By lemma~\ref{lemma:DerBded}, $\eta^{v^l}(t;t_{0},x) \to \eta^v(t;t_{0},x)$ as $l \to \infty$ (in
the Euclidean norm) for all $t, t_{0} \in [0,1]$ and $x \in \Omega$. Thus,  the time-one maps generated by $v^{l}$ and
their inverses converge pointwise: $\phi^{v^l}(x) \to \phi^v(x)$ and $(\phi^{v^l})^{-1}(x) \to (\phi^v)^{-1}(x)$ as
$l \to \infty$ for all $x \in \Omega$. Because the manifolds $M$ and $N$ are
compact, $\phi^v(M)\subset N$ and
$(\phi^v)^{-1}(N) \subset M$. In view of these inclusions,  the $C^{r}$ diffeomorphism  $\phi^v$ of $\Omega$ restricted to $M$ is a diffeomorphism, that is,  $\psi^v=\phi^v|_M \in \diff^{r}(M,N)$.

By passing to the limit as $l\to \infty$ in the inequality $\|v\|_{\SH^k}^{2} \leq \langle v-v^{l},
v\rangle+\AdmSetConst\|v\|_{\SH^k}$, it follows that $\|v\|_{\SH^k}\leq \AdmSetConst$. Therefore $v \in \SA^k_\AdmSetConst$,  as required.

(ii) For simplicity, let us assume that $s=2$. 
Let $(U,\xi)$ be a chart on $M$ at $p_0$. 
It suffices to show that 
\begin{equation}
\label{eq:ComponentsConverge}
B^{l}:=(\phi^{v^l})^\ast b(X,Y)(p_0)-(\phi^{v})^\ast b(X,Y)(p_0)\to 0
\end{equation}
as $l \to \infty$ for all smooth vector fields $X, Y$ on $U$.

Using the notation 
\begin{eqnarray*}
q^l&=&\phi^{v^l}(p_0),\\
q&=&\phi^{v}(p_0),\\ 
Z^l_y&=&D\phi^{v^l}X\circ (\phi^{v^l})^{-1}(y), \\
Z_y&=&  D\phi^{v}X\circ (\phi^{v})^{-1}(y),\\
Q^l_y&=& D\phi^{v^l}Y\circ (\phi^{v^l})^{-1}(y), \mbox{ and }\\
Q_y&=&D\phi^{v}Y\circ (\phi^{v})^{-1}(y)
\end{eqnarray*}
for all $y \in N$, 
the quantity $B^{l}$ in expression~\eqref{eq:ComponentsConverge} is recast in the form
\begin{eqnarray}
\label{eq:Difference1}
\nonumber
B^{l}&=& b(q^l)(Z^l_{q^l}, Q^l_{q^l})-b(q)(Z_q, Q_q)\\
&=& 
b(q^l)(Z^l_{q^l}, Q^l_{q^l})
- b(q^l)(Z_{q^l}, Q^l_{q^l})\\
\label{eq:Difference2}
&&{}+b(q^l)(Z_{q^l}, Q^l_{q^l})
-b(q^l)(Z_{q^l}, Q_{q^l})\\
\label{eq:Difference3}
&&{}+b(q^l)(Z_{q^l}, Q_{q^l})
-b(q)(Z_q, Q_q).
\end{eqnarray}

Using definition~\ref{df:MultilinearNorm} and noting that the Riemannian metric $g_N$ is inherited from $\R^{n+1}$, we estimate difference~\eqref{eq:Difference1} as follows:
\begin{eqnarray*}
\label{est:Difference1}
b(q^l)(Z^l_{q^l}, Q^l_{q^l})
- b(q^l)(Z_{q^l}, Q^l_{q^l})\leq \|b(q^l)\|_{g_{N}(q_{l})} |Z^l_{q^l}-Z_{q^l}|_{\R^{n+1}}|Q^l_{q^l}|_{\R^{n+1}}.
\end{eqnarray*}
By lemma~\ref{lemma:DerBded},  $|Z^l_{q^l}-Z_{q^l}|_{\R^{n+1}} \to 0$ as $l \to \infty$, and $|Q^l_{q^l}|_{\R^{n+1}}$ is uniformly bounded in $l \in \N$. 
By lemma~\ref{lemma:MUltNormBded}, there exists a constant $C>0$ such that $\|b(q^{l})\|_{g_{N}(q^{l})}\leq C$ for all
$l \in \N$. Therefore, difference~\eqref{eq:Difference1} converges to zero as $l\to \infty$.
Similarly, it can be shown that difference~\eqref{eq:Difference2} converges to zero as $l\to \infty$.  Difference~\eqref{eq:Difference3} converges to zero as 
$l \to \infty$ because $z \mapsto b(z)(Z_z,Q_z)$ is a continuous function on $U$.
Hence,  $B^{l}\to 0$ as $l \to \infty$.
\end{proof}

We say that a functional $I:\SH^k\to\R$ is weakly continuous on $\SH^k$ if $I(v^l)\to I(v)$
whenever the sequence $\{v^l\}_{l=1}^\infty\subset\SH^k$ converges weakly to $v$ in $\SH^k$.

Recall that the inequality $k \geq (n+1)/2+r+1$ guarantees the embedding of
the Sobolev space $W^{k,2}_0(\Omega,\R^{n+1})$ into
$C^r(\bar{\Omega},\R^{n+1})$, where $r\geq 1$.
\begin{lemma}
\label{l123}
Assume that the constant $\AdmSetConst>0$ is large enough so that the set
$\SA^k_\AdmSetConst$ is not empty. Let the functionals $J, I_1, I_2$ be defined 
as in definition~\ref{df:I12}.\\
(i)
If $k \geq (n+1)/2+3$, then 
the functionals $J:\SA^k_\AdmSetConst \to \R_{+}$ and $I_1:\SA^k_\AdmSetConst \to
\R_{+}$ are weakly continuous.\\
(ii)
If $k \geq ({n+1})/{2}+4$, then
the functional $I_2:\SA^k_\AdmSetConst \to \R_{+}$  is weakly continuous.
\end{lemma}
\begin{proof}
Let $\{v^l\}_{l=1}^\infty \subset \SA^k_\AdmSetConst$ and suppose that $v^l$ converges weakly
to some $v \in {\SH^k}$ as $l \to \infty$ (in symbols $v^l \wto v \in {\SH^k}$). 
By lemma~\ref{lemma:PullBackConvergence}, $v\in \SA^k_\AdmSetConst$ and $J(v)$, 
$I_1(v)$, and $I_2(v)$ are well-defined.

\noindent (i) We will show that $\lim_{l \to \infty}J(v^l)= J(v)$. 

Let $G:=g_M^\ast \otimes g_M^\ast$. For  tensor fields
$a,b \in \ST^{0}{}_{2}(M)$ and every $p \in M$, we have the equality
$$| \|a\|^2(p)-\|b\|^2(p)|=|G(a+b,a-b)(p)|\leq\|a+b\|(p)\|a-b\|(p).$$  
By applying the Cauchy-Schwarz inequality, we obtain the inequality
\begin{eqnarray}
|J(v^l)-J(v)|&\leq& \int_M\|(\psi^{v^l})^\ast \tau_N+(\psi^{v})^\ast \tau_N-2\tau_M\|
\|(\psi^{v^l})^\ast \tau_N-(\psi^{v})^\ast \tau_N\|\w_M
\nonumber
\\
\nonumber
&\leq&
\big( \int_M\|(\psi^{v^l})^\ast \tau_N+(\psi^{v})^\ast \tau_N-2\tau_M\|^2\w_M\big)^{1/2}\\
\label{Jineq}
&&{}\times
\big(\int_M\|(\psi^{v^l})^\ast \tau_N-(\psi^{v})^\ast \tau_N\|^2\w_M\big)^{1/2}.
\end{eqnarray}
By lemma~\ref{lemma:PullBackConvergence}, 
\begin{equation}
\label{eq:PBConvergence}
\lim_{l\to \infty}\|(\psi^{v^l})^\ast \tau_N-(\psi^v)^\ast \tau_N\|^2(p)= 0
\end{equation}
for all $p \in M$.

Let $K>0$ be the constant in display~\eqref{eq:DerBded} of lemma~\ref{lemma:DerBded}. By
lemma~\ref{lemma:MUltNormBded} and because the manifold $N$ is compact, there exists
a constant $C>0$ such that $\|\tau_{N}(z)\|_{g_{N}(z)}\leq C$ for all $z \in N$.
Using the equivalence of norms~\eqref{eq:NormsEquiv}, we estimate
\begin{eqnarray}
\nonumber
\|(\psi^{v^{l}})^{\ast}\tau_{N}\|(p) &\leq & n^{s/2}\|(\psi^{v^{l}})^{\ast}\tau_{N}(p)\|_{g_{M}(p)}\\
\nonumber
&\leq& n^{s/2}\|\tau_{N}(\psi^{v^{l}}(p))\|_{g_{N}(\psi^{v^{l}}(p))}|D\psi^{v^{l}}(p)|^{s}\\
\label{eq:PullBackBded}
&\leq& n^{s/2} C K^s.
\end{eqnarray}

Using inequalities~\eqref{Jineq} and ~\eqref{eq:PullBackBded}, limit~\eqref{eq:PBConvergence}, and the Dominated
Convergence Theorem, we conclude that $J(v^l) \to J(v)$ as $l \to \infty$.

Let us show that the functional $I_1$ is weakly continuous. By an estimate analogous to~\eqref{Jineq}, it suffices to prove the following statements.
\begin{itemize}
\item[(I)] If $p \in M$ and $t \in [0,1]$, then
$$
\lim_{l \to \infty}\|F^{v^l}(\cdot, t)^\ast g^{v^l}_t-F^{v}(\cdot, t)^\ast g^v_t\|^2(p) = 0.
$$
 \item[(II)] There exists $S_1>0$ such that 
$$
\|F^{v^l}(\cdot,t)^\ast g^{v^l}_t\|^2(p) \leq S_1
$$ 
for all $p \in M$, $t \in [0,1]$, and $l \in \N$.
\end{itemize}

Because all the Riemannian metrics are inherited from $\R^{n+1}$, whose standard inner product is denoted by $\langle\cdot,\cdot\rangle$,
we have
\begin{eqnarray*}
 F^{v^l}(\cdot,t)^\ast g_t^{v^l}(p)(X,Y)-F^v(\cdot,t)^\ast g_t^v(p)(X,Y)
&=&
\langle D_xF^{v^l}(p,t)X,D_xF^{v^l}(p,t)Y \rangle\\
&&{}-\langle D_xF^{v}(p,t)X,D_xF^{v}(p,t)Y\rangle,
\end{eqnarray*}
for all $p \in M$, $X, Y \in T_{p}M$, and $t \in [0,1]$,
where $D_{x}$ denotes the derivative with respect to the spatial variable.
The right-hand side of this equation converges to zero as $l \to \infty$ by lemma~\ref{lemma:DerBded}.
This completes the proof of statement~(I).

By the same lemma and inequality~\eqref{eq:NormsEquiv}, for every $p \in M$ we have
\begin{eqnarray*}
\|F^{v^{l}}(\cdot,t)^\ast g_{t}^{v^{l}}\|(p)&\leq& n \|F^{v^{l}}(\cdot,t)^{\ast}g_{t}^{v^{l}}(p)\|_{g_{M}(p)}\\
&\leq&n |D_{x}F^{v^{l}}(p,t)|^{2}\\
&\leq&n K^{2}.
\end{eqnarray*}
This inequality implies statement (II).

(ii) We will show the weak continuity of the functional $I_2$.
By an estimate analogous to~\eqref{Jineq}, it suffices to show two facts:
\begin{itemize}
\item[(III)] If $p \in M$ and $t \in [0,1]$, then
$$
\lim_{l\to \infty}\|F^{v^l}(\cdot, t)^\ast \II^{v^l}_t-F^{v}(\cdot, t)^\ast \II^v_t\|^2(p) = 0.
$$
\item[(IV)] There exists $S_2>0$ such that 
$$
\|F^{v^l}(\cdot,t)^\ast \II^{v^l}_t\|^2(p) \leq S_2
$$
for all $p \in M$, $t \in [0,1]$ and $l \in \N$.
\end{itemize}

We will first prove statement~(IV).

Consider a morph $F^w$ generated by a time-dependent vector field 
$w\in \SA^k_\AdmSetConst$.
By definition of morphs of class $\SM^{r,\text{ac}}(M,N)$ in section~\ref{subs:MPDR}, the orientation of each intermediate manifold 
$M^{w,t}$, where $t \in [0,1]$, is induced by the $C^2$
diffeomorphism $F(\cdot,t):M \to M^{w,t}$. 
Let $\SN^{w,t}(z)$ denote the unit normal to the  manifold $M^{w,t}$ at the
point $z \in M^{w,t}$. We assume that for every positively oriented basis
$\{X_i\}_{i=1}^n$ of $T_zM^{w,t}$, the set of vectors $\{X_1,\ldots,X_n,\SN^{w,t}(z)\}$ is positively oriented in $\R^{n+1}$.

For $p\in M$, let $(U,\xi)$ be a chart at $p$, choose two smooth vector fields $X$ and $Y$ on $U$, and let
$\gamma:[0,1]\to U$ be a
$C^1$ curve at $p$ such that $\dot{\gamma}(0)=X_{p}$.
It is evident that the inner product 
\begin{equation}
\label{eq:InnProd1}
\langle \SN^{w,t}(F^{w}(\gamma(s),t)),D_xF^{w}(\gamma(s),t) Y_{\gamma(s)}\rangle=0
\end{equation}
for every $t,s \in [0,1]$.
Let us recall that for every $t \in [0,1]$ the function $x\mapsto F^{w,t}(x,t)$ is defined for all $x \in \Omega$
and denote its second derivative at $x\in \Omega$ by $D^2_xF(x,t)$.
By differentiating expression~\eqref{eq:InnProd1} with respect to $s$ at $s=0$, we obtain the equality
\begin{eqnarray}
\nonumber
F^w(\cdot,t)^\ast\II^w_t (p)(X_{p},Y_{p})&=&\langle \bar{\nabla}_{D_xF^{w}(p,t)X_p}\SN^{w,t}(F^{w}(p,t)),D_xF^{w}(p,t)Y_{p}\rangle\\
\label{eq:2FFormPB}
&=&
-
\langle \SN^{w,t}(F^{w}(p,t)),D_x^2F^{w}(p,t)[X_p,Y_p]\rangle,
\end{eqnarray}
where $\bar{\nabla}$ denotes the standard Riemannian connection on  $\R^{n+1}$ (see~\cite{H}).

For every $p \in M$, let $W_{p}, Q_{p}\in T_{p}M$ be unit length vectors such that
$$\|F^{v^l}(\cdot,t)^\ast\II^{v^l}_t(p)\|_{g_{M}(p)}=|F^{v^l}(\cdot,t)^\ast\II^{v^l}_t(p)(W_{p},Q_{p})|.$$
Using inequality~\eqref{eq:NormsEquiv} and equation~\eqref{eq:2FFormPB}, we have the estimates
\begin{eqnarray*}
\label{eq:PBBded}
\|F^{v^l}(\cdot,t)^\ast\II^{v^l}_t\|(p)&\leq& n \|F^{v^l}(\cdot,t)^\ast\II^{v^l}_t(p)\|_{g_{M}(p)}\\
&=&n |F^{v^l}(\cdot,t)^\ast\II^{v^l}_t(p)(W_{p},Q_{p})|\\
&=&n |\langle \SN^{v^{l},t}(F^{v^{l}}(p,t)),D_x^2F^{v^{l}}(p,t)[W_p,Q_p]\rangle|\\
&\leq&n K.
\end{eqnarray*}
This completes the proof of statement~(IV).

By lemma~\ref{lemma:DerBded}, if $\al\in \{0,1,2\}$, then the derivative 
$D_x^\al F^{v^l}(p,t)$ converges to $D_x^\al F^{v}(p,t)$ as $l \to \infty$ in the Euclidean norm for every $p\in M$ and $t\in [0,1]$.  Taking into account equation~\eqref{eq:2FFormPB},
we see that statement~(III) follows from the convergence 
\begin{equation}
\label{eq:NormalConv}
\SN^{v^l,t}(F^{v^l}(p,t)) \to\SN^{v,t}(F^{v}(p,t))
\end{equation}
as $l\to \infty$ in $\R^{n+1}$ for every $p\in M$ and $t \in [0,1]$. 
\end{proof}

\begin{theorem}
\label{th123}
Assume that the constant $\AdmSetConst>0$ is large enough so that the set
$\SA^k_\AdmSetConst$ is not empty.\\
(i) If $k\geq ({n+1})/{2}+3$, then there exists a minimizer of the bending 
distortion energy functional $E$ in the admissible set $\SA^k_\AdmSetConst$.\\
(ii) If $k\geq ({n+1})/{2}+4$, then there exists a minimizer of the morphing 
distortion energy functional $\SE$ in the admissible set $\SA^k_\AdmSetConst$.
\end{theorem}
\begin{proof}
Let $\{v^{l}\}_{l=1}^{\infty}\subset \SA^k_\AdmSetConst$ be a minimizing sequence of $E$, that is   $$\lim_{l \to \infty}E(v^{l})=\inf_{w\in
\SA^k_\AdmSetConst}E(w)\geq 0.$$  By
lemma~\ref{lemma:PullBackConvergence}, the set $\SA^k_\AdmSetConst$ is sequentially weakly closed and bounded. Therefore, there exists
a weakly
convergent subsequence $\{v^{l_{k}}\}_{k=1}^{\infty}$ with the weak limit $v \in \SA^k_\AdmSetConst$. The functional $E$ is weakly
continuous by lemma~\ref{l123}. Therefore, $E(v)=\inf_{w\in \SA^k_\AdmSetConst}E(w)$ and $v$ is a minimizer of $E$.

The existence of minimizers for the functional $\SE$ is proved in the same fashion.
\end{proof}

\begin{remark}
Theorem~\ref{th123} implies the existence of minimizers of the functional $\Lambda$ defined in display~\eqref{eq:Funct2} in the admissible set
$$\SB^k_P:=\{h \in \diff^2(M,N): h=\psi^v \text{ for some } v\in \SA^k_P\}.$$
The set $\SB^k_P$, among other maps, contains smooth diffeomorphisms $f:M \to N\subset \R^{n+1}$ that are homotopic to the inclusion map 
$i:M\to \R^{n+1}$ and generate time-dependent vector fields in $\SA^k_P$.

To minimize the distortion energy of
diffeomorphisms from other isotopy classes,  we consider the family of maps $\{\psi^v\circ \phi \in \diff^r(M,N): v\in \SA^k_P\}$,
where $\phi$ is a fixed diffeomorphism of $M$.
Similarly, given a smooth isotopy $G:[0,1]\times M\to M$, we consider the family of morphs $\{F^v_G\in 
\SM^{r,\text{ac}}(M,N): v \in \SA^k_P\}$, where 
$F^v_G(p,t)=F^v(G(p,t),t)$ for all $(p,t)\in M\times [0,1]$, 
as candidates for minimal distortion morphs. The most interesting example of this generalization is, perhaps, 
the case where $G(p,t)=\phi(p)$
for some fixed diffeomorphism $\phi:M \to N$, so that the admissible isotopies are from
the class of morphs $F^v(\phi(p),t)$ generated by time-dependent vector fields in $\SA^k_P$, where $p \in M$ and $t \in [0,1]$.

The latter idea leads to the definition of the functionals
\begin{eqnarray*}
 E_\phi(v)=E_\phi(v;B_1,B_2)&=& B_1 \int_M\|(\psi^v\circ \phi)^\ast g_N-g_M\|^2\w_M\\
&&{}+B_2 \int_M\|(\psi^v\circ \phi)^\ast \II_N-\II_M\|^2\w_M
\end{eqnarray*}
and
\begin{eqnarray*}
\SE_G(v)=\SE_G(v;B_1,B_2)&=&B_1 \int_0^1\int_M\|F^v_G(\cdot,t)^\ast g^v_t-g_M\|^2\,\w_M dt\\
&&{}+B_2 \int_0^1\int_M\|F^v_G(\cdot,t)^\ast \II^v_t-\II_M\|^2\,\w_M dt,
\end{eqnarray*}
where $B_1$ and $B_2$ are nonnegative real numbers
(cf. definition~\ref{df:ESE}), $\phi \in \diff(M)$ and $G:M\times[0,1]\to M$ is an isotopy. .
\end{remark}

Theorem~\ref{th123} can be easily generalized to show that for $P>0$ and $k \in \N$ sufficiently large, both functionals $E_\phi$ and $\SE_G$ have minimizers in $\SA^k_P$ for
every diffeomorphism $\phi:M\to M$ and isotopy $G:M\times[0,1]\to M$. 
\begin{theorem}
\label{phith123}
Assume that the constant $\AdmSetConst>0$ is large enough so that the set
$\SA^k_\AdmSetConst$ is not empty. Let $\phi \in \diff(M)$ and
let $G:M\times[0,1]\to M$ be an isotopy.\\
(i) If $k\geq ({n+1})/{2}+3$, then there exists a minimizer of the bending 
distortion energy functional $E_\phi$ in the admissible set $\SA^k_\AdmSetConst$.\\
(ii) If $k\geq ({n+1})/{2}+4$, then there exists a minimizer of the morphing 
distortion energy functional $\SE_G$ in the admissible set $\SA^k_\AdmSetConst$.\\
(iii) If $k\geq ({n+1})/{2}+3$, then both functionals
$\Phi$ and $\Lambda$ defined in displays~\eqref{eq:Funct1} and~\eqref{eq:Funct2} respectively have minimizers in the admissible set
$$\SB^k_{P,\,\phi}:=\{h \in \diff^2(M,N): h=\psi^v\circ \phi \text{ for some } v\in \SA^k_P\}.$$
\end{theorem}

The latter theorem is an easy generalization of theorem~\ref{th123}.
More precisely, let $\{b^l\}_{l=1}^\infty$ be a sequence of tensor fields in 
$\ST^0{}_s(M)$ such that  $\lim_{l\to \infty}\|b^l\|(p)=0$ and 
$\|b^l(p)\|\leq K$ for all $p\in M$ and $l\in \N$, where $K$ is a positive 
constant and
 let $\phi\in \diff(M)$. Then $\lim_{l\to \infty}\|\phi^\ast b^l\|(p)=0$ 
and there exists a constant $K_1>0$ such that $\|\phi^\ast b^l(p)\|\leq K_1$ 
for all $p\in M $ and $l\in \N$. Using the above observation, lemma~\ref{l123} is 
easily generalized to the case where $\psi^v$ and $F^v$ are replaced with $\psi^v\circ \phi$ and $F^v_G$ respectively, and the proof of 
theorem~\ref{th123} remains the same.

In theorem~\ref{phith123}, the statement~(iii), which is equivalent to theorem~\ref{MainThm2}, follows from the statement (i).

\section{A Minimal Distortion  Morph}
\label{Sec:Example}
We have proved the existence of minimizers of the 
functionals $E$ and $\SE$, which produce 
minimal distortion diffeomorphisms and morphs between manifolds $M$ and $N$. 
In this section, we consider the special case where  $M=\Sph^1$ is the unit circle in the plane and $N=\Sph^1_R$  is the concentric circle of 
radius $R>1$ and construct a minimal distortion 
diffeomorphism and morph between them.

Our example of a minimal distortion morph in subsection~\ref{ex:MinDistMorph} demonstrates the importance of the 
bound $\|v\|_{{\SH^k}}\leq \AdmSetConst$ in the definition of the admissible set $\SA^k_\AdmSetConst$. 
If this bound is not imposed, there is 
a minimizing sequence of morphs $\{F_n\}_{n=1}^\infty$ such that the 
distortion energy 
\begin{eqnarray}
\nonumber
\Psi(F_n)&:=&\int_0^1 \int_M\|F_n(\cdot,t)^\ast g_t^n-g_M\|^2 \,\w_M\,dt\\
\label{eq:Psin}
&&{}+
\int_0^1 \int_M\|F_n(\cdot,t)^\ast \II_t^n-\II_M\|^2 \,\w_M\,dt
\end{eqnarray}
tends to 
zero, where $g_t^n$ and $\II_t^n$ are the first and the second fundamental forms
of the intermediate manifold $F_n(M,t)$ induced by its embedding into $\R^2$.
An example of such a sequence is $F_n(p,t)=\phi_n(t)p$ for all 
$t \in[0,1]$ and $p\in M$, where $\phi_n\in C^\infty(0,1)\cap C[0,1]$ is a 
function whose values remain in the segment $[1,R]$ and such that $\phi_n(t)=1$ for all $t \in[0,1-1/n]$ and
$\phi_n(1)=R$. 
From the representation of~\eqref{eq:Psin} in local coordinates (see~\eqref{eq:PsiLocal}) we derive
\begin{eqnarray}
\label{eq:phinSeq}
\Psi(F_n)&=&2\pi \int_0^1\big[(\phi_n^2-1)^2+(\phi_n-1)^2\big]\,dt\\
\nonumber
&\leq& 2 \pi [(R^2-1)^2+(R-1)^2]\frac{1}{n};
\end{eqnarray} 
hence, $\lim _{n \to \infty}\Psi(F_n)=0$. On the other hand, there is no morph
$H$ in the space $\SM^{r,\text{ac}}(M,N)$ with $r>1$
such that $\Psi(H)=0$: otherwise, $H(\cdot,1)$ would be an isometry between $M$
and $N$.
The sequence $\{F_n\}_{n=1}^\infty$ converges pointwise to the discontinuous morph
$$F(p,t)=
\left\{
\begin{array}{cl}
p,\quad & \mbox{ if } 0\leq t<1,\\
Rp, \quad &\mbox{ if } t=1
\end{array}
\right.
$$
whose distortion energy $\Psi(F)$
vanishes.

Theorem~\ref{th123} implies that every sequence of time-dependent 
vector fields $\{v^n\}_{n=1}^\infty\subset {\SH^k}$  such that each 
$v^n \in {\SH^k}$ generates the morph $F_n$ must be unbounded in ${\SH^k}$. 

In our example of a minimal distortion morph, we solve the optimization 
problem for the minimal distortion morph between 
$\Sph^1$ and $\Sph^1_R$ in the class of morphs,  whose intermediate states
are circles of increasing radii, that are generated by time-dependent vector fields whose norms are  
uniformly  bounded by a positive constant $\AdmSetConst$. 
Numerical solutions suggest  that the second time-derivative 
$\pd^2 F/\pd t^2$ of the minimal morph $F$ increases as $\AdmSetConst$ increases. In effect, the
choice of the constant $\AdmSetConst$ in the definition of the admissible set $\SA^k_\AdmSetConst$
sets a restriction on the magnitude of the curvature of the curves
$t\mapsto F(p,t)$, where $p \in M$.

We begin the construction of the minimal distortion morph with the example of 
a minimal distortion diffeomorphism between  
 $\Sph^1$ and $\Sph^1_R$. This example 
is based on the theory of minimal deformation (as measured by the functional
$\Phi$) bending 
of regular simple closed curves developed in~\cite{OpusculaPaper}. 

\subsection{A Minimal Distortion Diffeomorphism between two Circles}
\label{sec:3.1}
We will construct a minimal distortion diffeomorphism between the circles
$M=\Sph^1$ and $N=\Sph^1_R$.

For $r\geq 1$, we consider the functional $\Lambda:\diff^r(M,N)\to \R_+$ 
defined in display~\eqref{eq:Funct2}.
Also, using the radius $R>1$ of $\Sph^1_R$, we define the radial map $h_R:\R^2\to\R^2$  by $h_R(p)=R p$.
\begin{lemma}
\label{cor:MinPhiIandII}
The restriction of the radial map $h_R$ to $\Sph^1$ minimizes the functional 
$\Lambda:\diff^2(M,N)\to \R_+$ defined in display~\eqref{eq:Funct2}.
\end{lemma}
\begin{proof}
Fix  $p \in M$ and $q\in N$, and let $\gamma:[0,L(M))\to M$ and $\xi:[0,L(N))\to N$ be the (positive orientation) arc length parametrizations of $M$ and $N$ respectively such that $\gamma(0)=p$ and 
$\xi(0)=q$. The distortion 
energy functional $\Lambda$ can be recast in the form
\begin{eqnarray}
\nonumber
\Lambda(u)&=&\int_0^{L(M)}(\dot{u}^2-1)^2\,dt+\int_0^{L(M)}(\frac{1}{R}\dot{u}^2-1)^2\,dt\\
\label{eq:PhiI1I2}
&=:&J_1(u)+J_2(u),
\end{eqnarray}
where $u=\xi^{-1}\circ h \circ \gamma:[0,L(M))\to [0,L(N))$ 
is a local coordinate representation of $h \in \diff(M,N)$ with $h(p)=q$. 
By lemma~4.1 in~\cite{OpusculaPaper}, the functions $u_1(t)=L(N)/L(M) t$ and 
$u_2(t)=-L(N)/L(M) t+L(N)$ minimize the functional $J_1$ in the admissible 
set \begin{eqnarray*}
\mathcal{B}=\big\{u \in C^2(0,L(M))\cap C([0,L(M)]): u \mbox{ is a
bijection onto } [0,L(N)]  \big\}.
\end{eqnarray*} 
The proof of the statement
that $u_1$ and $u_2$ minimize the functional $J_2$ in $\mathcal{B}$ follows 
along the same lines.

Therefore, the map $h_R|_{\Sph^1}$ minimizes the functional~\eqref{eq:Funct2}
over the set of all maps $h \in \diff^2(M,N)$ such that, for our fixed 
$p\in M$, $h(p)=R p$. 

If $h\in \diff^2(M,N)$ is such that $h(p)=q\neq R p$, consider an isometry
$f:N\to N$ such that $f(q)= R p$.
Because $f^\ast g_N=g_N$ and $f^\ast\II_N=\II_N$, we obtain
$\Lambda(h)=\Lambda(f\circ h)\geq \Lambda(h_R|_{\Sph^1})$, which proves the
lemma.
\end{proof}

As before, let $\psi^v$ denote the time-one map of the vector field $v
\in \SA^k_\AdmSetConst$ restricted to $M$.
Using lemma~\ref{cor:MinPhiIandII}, it is easy to construct a time-dependent vector 
field that minimizes the functional $E(v)=\Lambda(\psi^v)$ in the admissible set 
$\SA^k_\AdmSetConst$. In fact, every vector field $v^0 \in \SA^k_\AdmSetConst$ that 
generates the time-one map $\phi^v$ such that its restriction to $M$ is 
$\psi^v=h_R|_M$, minimizes the functional $E$. An example of such a vector 
field is $$v(x,t)=\rho(x)w(x,t)$$ for all $x$ in the open ball 
$\Omega:=B(0,R+2)\subset \R^2$ and $t \in [0,1]$, where 
$$w(x,t)=\frac{R-1}{1+(R-1)t}x$$ and $\rho:\R^2 \to \R$ is a bump function such 
that $0\leq\rho\leq 1$, $\rho\equiv 1$ in the open ball $B(0,R+1)\subset \R^2$,
 and $\rho\equiv0$ on $\Omega^c$. The vector field
$v$ generates the morph $F^v(p,t)=(1+(R-1)t)p$, whose time-one map restricted to
$M$ is $\psi^v(p)=Rp$.

\subsection{A Minimal Distortion Morph\\ between two Circles}
\label{ex:MinDistMorph}
Let us assume, as before, that $M=\Sph^1$, $N=\Sph^1_R$, and $R>1$.

In the previous subsection, we have constructed a minimizer of the functional 
E; the construction was quite straight-forward. 
The
time-integral involved in the definition of the functional $\SE$ makes 
the construction of its minimizer a much more intricate process.
 We will restrict our attention to morphs that operate through images that are
concentric circles, while leaving open the question whether a
minimizer must be purely radial, as the problem of constructing a minimal morph
within this family is difficult enough.  Note that
although our functional is formally defined in terms of time-dependent
vector fields, it is the resulting morphs we will be working with directly.

We will construct a minimal distortion morph between the  circles
$M=\Sph^1$ and $N=\Sph^1_R$ in case $R>1$. 
As before, let $\SM^{3,\text{ac}}(M,N)$ be the class of morphs between the manifolds $M$ and $N$ that are absolutely continuous in time and class $C^3$ in the spatial variable.
Recall that for a morph $F \in \SM^{3,\text{ac}}(M,N)$ we define 
$f^t=F(\cdot,t) \in \diff^3(M,M^t)$. We assume that the morph $F$ is generated
by a time-dependent vector field $v \in \SA^k_\AdmSetConst$.

Consider the functional $\Psi:\SM^{r,\text{ac}}(M,N) \to \R_+$, where $r\geq 1$, defined by
$$
\Psi(F)=\int_0^1\int_M \|(f^t)^\ast g_t-g_0\|^2\,\w_M \,dt+
\int_0^1\int_M \|(f^t)^\ast \II_t-\II_0\|^2\,\w_M \,dt,
$$
where $g_t$ and $\II_t$ are the first and the second fundamental forms on the
intermediate state $M^t$ induced by its isometric embedding into $\R^2$.
We notice that 
$\SE(v;1,1)=\Psi(F^v)$ for all $v\in\SA^k_\AdmSetConst$ 
(see definition~\ref{df:ESE}).

Fix a point $p\in M$. Let $\gamma$ be an arc-length parametrization of 
$M$ that induces the positive orientation on $M$ with $\gamma(0)=p$. 
Let $\xi^t$  be the arc length reparametrization of $M^t$ obtained from the 
parametrization $f^t\circ\gamma$ such that $\xi^t(0)=f^t\circ\gamma(p)$
and both $\xi^t$ and $f^t\circ\gamma$ induce the same orientation of $M^t$. Such a parametrization can be obtained by solving the equation $s(t,x)=y$ for $x$, 
where $s(t,x)=\int_0^x|f^t\circ\gamma(\tau)|\,d\tau$ is the arc length function of
the curve $M^t$. Using the implicit solution $x(t,y)$  of $s(t,x)=y$, we define 
$\xi^t(y)=f^t\circ\gamma\circ x(t,y)$. Because the morph $F$ is generated by a
time-dependent vector field $v \in \SA^k_\AdmSetConst$, lemma~\ref{lemma:Diff} implies that the function 
$t\mapsto D f^t(p)$, where $p\in M$, is absolutely continuous. It follows that 
the function $t\mapsto \xi^t(s)$ is continuous for every $s \in [0,L(M^t))$. 

The local representation of $f^t$ is given by $u^t(s)=(\xi^t)^{-1}\circ f^t \circ \gamma(s)$, where $s \in [0, 2\pi]$, and the energy $\Psi(F)$ of the morph $F$ is
\begin{eqnarray}
\nonumber
\Psi(F)&=&\int_0^1\int_0^{2\pi} 
\Big(\Big(\frac{du^t}{ds}\Big)^2-1\Big)^2
\,ds\,dt\\
&&{}
\label{eq:PsiLocal}
+
\int_0^1\int_0^{2\pi} 
\Big(\curvature_t(u^t)\Big(\frac{du^t}{ds}\Big)^2-1\Big)^2
\,ds\,dt,
\end{eqnarray}
where $\curvature_t:[0,L(M^t)]\to \R$ is the curvature function of the intermediate state $M^t$.

Let us restrict our attention to the morphs whose intermediate states are 
circles of increasing radii such that
each intermediate state $M^t$  of such  a morph $F$  is a circle of radius 
$\psi(t)$ with  $\psi \in C^2(0,1)\cap C[0,1]$  a (strictly) 
increasing function. In symbols, 
$$
\psi \in Q_+:=\{\phi \in C^2(0,1)\cap C[0,1]: \phi(0)=1, \phi(1)=R, 
\mbox{ and } \phi \mbox{ is increasing}\}.
$$
The curvature function 
of $M^t$ is given by $\curvature_t\equiv 1/{\psi(t)}$. By lemma~\ref{cor:MinPhiIandII}, the radial 
map between the circles $M$ and $M^t$ minimizes the functional 
$$f^t \mapsto \int_M \|(f^t)^\ast g_t-g_0\|^2\,\w_M+\int_M \|(f^t)^\ast
\II_t-\II_0\|^2\,\w_M.$$
Therefore, 
$$
\Psi(F)\geq \Psi(H)=2 \pi \int_0^1 (\psi^2-1)^2\,dt+2 \pi \int_0^1 (\psi-1)^2\,dt,
$$
where the morph $H\in \SM^{\infty,2}(M,N)$ is given by $H(p,t)=\psi(t) p$.

To determine the morph $H(p,t)=\psi(t) p$ of smallest distortion energy $\Psi(H)$, we 
will minimize the functional 
$J:L^4(0,1)\to \R_+$ defined by
\begin{equation}
\label{eq:JFunct}
J(\psi):= \int_0^1 (\psi^2-1)^2\,dt+\int_0^1 (\psi-1)^2\,dt
\end{equation}
over all admissible radius functions $\psi$.
To define the admissible set for the functional $J$, let 
us put this example into the context of time-dependent vector fields. 

Let $\Omega$ be the open ball of radius $R+2$ in $\R^2$.
Given a morph $H(p,t)=\psi(t) p$ (where $\psi \in Q_+$, $t \in [0,1]$, and 
$p \in M$), let us construct a time-dependent vector field 
$v \in {\SH^5}=L^2(0,1;W^{5,2}_0(\Omega;\R^2))$ that generates $H$, where the number of
weak derivatives $k=5$ is chosen in view of condition~(ii) of theorem~\ref{th123}.

Consider the class of morphs of the plane $\R^2$ that have the form
$F(x,t)=\psi(t)x$, where $\psi \in Q_+$. Define a time-dependent vector field 
$\bar{v}:\R^2\times [0,1] \to \R^2$ by 
$$\bar{v}(F(x,t),t)=\frac{\pd F}{\pd t}(t,x)$$ or, equivalently,
$$\bar{v}(x,t)=\frac{\psi'(t)}{\psi(t)}x.$$

Clearly, the morph $F$ satisfies the differential equation 
$dq/dt=\bar{v}(q,t)$. To obtain the required vector field $v$, 
multiply $\bar{v}$ by a bump function $\rho:\R^2 \to \R^2$ such that 
$\rho \equiv 1$ on the ball $B(0,R+1)$, $\rho\equiv 0$ on $\Omega^c$, 
and $0\leq \rho \leq 1$. The vector field
$$
v(x,t)=\frac{\psi'(t)}{\psi(t)}\rho(x)x
$$
belongs to the Hilbert space ${\SH^k}$ and generates the morph $$H(p,t):=\psi(t)
p=F|_{ M\times[0,1]}(p,t)$$ for all $(p,t)\in M\times [0,1]$.

In theorem~\ref{th123}, we require the admissible set $\SA^k_\AdmSetConst$, for some fixed $\AdmSetConst>0$,  to contain all vector fields $v\in {\SH^k}$ such that the norm of $v$ is bounded by  $\AdmSetConst$ and  $v$ generates a morph between the manifolds $M$ and $N$. 

Therefore, in addition to the assumption that $\psi \in Q_+$, we must assume 
that the time-dependent vector fields of the form  $v(x,t)=\frac{\psi'(t)}{\psi(t)}\rho(x)x$ are bounded in ${\SH^k}$ by a fixed constant $\AdmSetConst>0$. In symbols, the required bound is  
$$
\|v\|_{{\SH^k}}^2=\|\rho\cdot \mbox{id}_{\Omega}\|^2_{W^{5,2}_0(\Omega;\R^2)}
\int_0^1\Big(\frac{\psi'}{\psi}\Big)^2\,dt\leq \AdmSetConst^2.
$$

After introducing the constant 
\begin{equation}
\label{eq:DfA}
A:=\frac{\AdmSetConst^2}{\|\rho\cdot \mbox{id}_{\Omega}\|^2_{W^{5,2}_0(\Omega;\R^2)}},
\end{equation} we obtain the constraint 
\begin{equation}
\label{eq:GFunct}
G(\psi):=\int_0^1\Big(\frac{\psi'}{\psi}\Big)^2\,dt-A\leq 0.
\end{equation}

Note that the functional $J$ can be written in the form
$$
J(\psi)=\int_0^1 u(\psi)\,dt,
$$
where the smooth function $u(s)=(s^2-1)^2+(s-1)^2$ is strictly increasing  on
$(1,\infty)$.

To find a morph $H(p,t)=\psi(t)p$ with $\psi \in Q_+$,  which has minimal
distortion among the morphs $F \in \SM^{3,\text{ac}}(M,N)$ whose intermediate states
are circles with increasing radii, we must solve the optimization problem 
\begin{equation}
\label{OptPr1}
\begin{array}{l}
\mbox{minimize } J(\psi)\\
\mbox{for } \psi \in Q_+=\{\phi \in C^2(0,1)\cap C[0,1]:\\ \phi(0)=1,\, \phi(1)=R, 
\mbox{ and } \phi \mbox{ is increasing}\}\\
\mbox{subject to }
G(\phi)\leq 0.
\end{array}
\end{equation}

The solution of problem~\eqref{OptPr1} is obtained using the following outline: We will consider the related optimization problem 
\begin{equation}
\label{OptPr2}
\begin{array}{lll}
\mbox{minimize }  J(\psi), \\
\psi\in Q^{1,2}:=\{\phi \in W^{1,2}(0,1): \phi(0)=1, \phi(1)=R\}\\
\mbox{subject to } G(\phi)\leq 0, 
\end{array}
\end{equation}
where (because every function $\psi \in Q^{1,2}$ is absolutely continuous  on
$[0,1]$) the boundary conditions in the definition of the set $Q^{1,2}$
are to be understood in the classical sense.
We will determine the unique minimizer $\psi$ of the optimization problem~\eqref{OptPr2} and 
show that $\psi$ is an increasing $C^2$ function. 
Because $Q_+\subset Q^{1,2}$, the  same function $\psi$ is the unique solution of 
optimization problem~\eqref{OptPr1}.

\begin{lemma}
\label{lm:psiPr2Ex1R}
There exists a unique solution $\psi \in Q^{1,2}$ of the optimization
problem~\eqref{OptPr2}. Moreover,
$1\leq \psi(t)\leq R$ for all $t \in[0,1]$. 
\end{lemma}

Lemma \ref{lm:psiPr2Ex1R} is proved using the direct method of the calculus of variations.
First, we prove the existence of a minimizer of the functional $J$ subject to
the constraint $G(\psi)\leq 0$ in the admissible
set $W^{1,4/3}(0,1)$ with the appropriate boundary conditions, and then we show that the minimizer is, in fact, in class
$W^{1,2}(0,1)$. The inequalities $1\leq \psi$ and $\psi\leq R$ are proved by contradiction using the  the cut-off functions
$h_1(t)=\max\{1,\psi(t)\}$ and $h_2(t)=\min\{R,\psi(t)\}$, which would yield smaller values of the functional $J$ than the minimizer.
 The details are given in Appendix~\ref{Append:C}.

\begin{lemma}
\label{lm:PsiNecCond}
If the constant $A$ in definition~\eqref{eq:GFunct} satisfies the inequality $A > (\log R)^2$ (see also equation~\eqref{eq:DfA})
and $\psi\in Q^{1,2}$ is the solution of the optimization problem~\eqref{OptPr2}, 
then there exists a constant $\lambda> 0$ such that
\begin{itemize}
\item[(i)] $\psi$ is a critical point of the functional 
$J+\lambda G$ over the space of variations $W^{1,2}_0(0,1)$, and
\item[(ii)] $G(\psi)=0$.
\end{itemize}
Moreover, the solution $\psi$ of the optimization problem~\eqref{OptPr2} is in
class $C^2(0,1)$.
\end{lemma}
Lemma \ref{lm:PsiNecCond} follows from the Generalized Kuhn-Tucker theorem 
(see theorem~1, Sec. 9.4 in~\cite{Luenberger}) and a regularity result for weak solutions of
Euler-Lagrange equations (see theorem~1.2.3, Sec.~1.2 in~\cite{JostCV}). The proof of the 
lemma is sketched in Appendix~\ref{Append:C} for completeness.

\begin{theorem}
If the constant $A$ in definition~\eqref{eq:GFunct} satisfies the inequality $A > (\log R)^2$,
then there exists a unique function $\psi \in C^2(0,1)\cap Q^{1,2}$ satisfying conditions~(i) and~(ii)
of lemma~\ref{lm:PsiNecCond} and the following properties.\\
(iii) The function $\psi$ is strictly increasing and
solves the initial value problem
\begin{equation}
\label{eq:psiIVP}
\left\{
\begin{array}{ll}
\psi'=\frac{1}{\sqrt{\la}}\psi \sqrt{\mu+(\psi^2-1)^2+(\psi-1)^2},\\
\psi(0)=1,
\end{array}
\right.
\end{equation}
where the pair of positive constants $\la$ and $\mu$ is the unique solution 
 of the system of equations
\begin{equation}
\label{eqHyp:CondLaE1}
\int_1^R \frac{ds}{s \sqrt{\mu+(s^2-1)^2+(s-1)^2}}=\frac{1}{\sqrt{\la}}
\end{equation}
and
\begin{equation}
\label{eqHyp:CondLaE2}
\frac{1}{\sqrt{\la}}\int_1^R \frac{\sqrt{\mu+(s^2-1)^2+(s-1)^2}}{s}\,ds=A.
\end{equation}
(iv)  The function $\psi$ is the unique solution
of the optimization problem~\eqref{OptPr1}.
\end{theorem}
\begin{proof}
If $\psi \in Z:=C^2(0,1)\cap Q^{1,2}$ is a critical point of the functional 
$J_\lambda:=J+\lambda G:W^{1,2}(0,1) \to \R_+$,  then $\psi$ satisfies the Euler-Lagrange equation for $J_\lambda$, which is equivalent 
to the Hamiltonian system
$$
\left\{
\begin{array}{ll}
\psi'=\frac{\pd H}{\pd p}(\psi,p),\\
p'=-\frac{\pd H}{\pd\psi}(\psi,p)
\end{array}
\right.
$$
with the Hamiltonian $H(\psi,p)=p \psi'-L(\psi,\psi')$, where 
$$L(\psi,\psi')=(\psi^2-1)^2+(\psi-1)^2+\la \Big(\frac{\psi'}{\psi}\Big)^2$$ is the integrand of $J_\la$ and 
$p:=\frac{\pd L}{\pd \psi'}(\psi,\psi')$ (see, for example, ~\cite{Chicone}). Moreover, the Hamiltonian $H(\psi,p)$ is constant along the solutions of the Euler-Lagrange
equation for $J_\la$. Let us denote this constant by $\mu$.

It is easy to see that $$p=2 \la \frac{\psi'}{\psi^2}$$ and the Hamiltonian is given by
$$
H(\psi,p)=\frac{1}{4 \la}p^2 \psi^2-(\psi^2-1)^2-(\psi-1)^2.
$$

Note that the equation $\psi'=\frac{\pd H}{\pd p}(\psi,p)$ yields  $\psi'=\frac{1}{2 \la}p \psi^2$. By solving the Hamiltonian energy equation
\begin{equation}
\label{eq:ppsiE}
\frac{1}{4 \la}p^2 \psi^2-(\psi^2-1)^2-(\psi-1)^2=\mu
\end{equation}
for $p$ and substituting, we obtain a first-order differential equation for $\psi$:
\begin{equation}
\label{eq:psiHamilt}
\psi'=\frac{1}{\sqrt{\la}}\psi \sqrt{\mu+(\psi^2-1)^2+(\psi-1)^2}\,.
\end{equation}
The case with the negative square root is eliminated because the
conditions $\psi(0)=1$ and $\psi(1)=R>1$ can be used to show  that the derivative of 
$\psi$  is non negative on $(0,1)$. 

In view of equation~\eqref{eq:ppsiE},  it is easy to see that 
$\mu+(\psi^2-1)^2+(\psi-1)^2\geq 0$ for all $\psi \in Z$. 
Because $\psi(0)=1$, we have $\mu\geq 0$. Also, it follows immediately from equation~\eqref{eq:psiHamilt} that $\psi$
is an increasing function. 

Let us use the notation $u(s)=(s^2-1)^2+(s-1)^2$ and recall that $u$ is a strictly increasing function on $(1,\infty)$. After integrating both sides of equation~\eqref{eq:psiHamilt} over the interval $0\le t\le 1$ and making the substitution $s=\psi(t)$, we obtain the relation
\begin{equation}
\label{eq:CondLaE1}
\int_1^R \frac{ds}{s \sqrt{\mu+u(s)}}=\frac{1}{\sqrt{\la}}.
\end{equation}

Another relation of $\la$ and $\mu$ is obtained from condition~(ii) in
lemma~\ref{lm:PsiNecCond} (see  equation~\eqref{eq:GFunct} for the definition of $G$). The integrand in the definition of $G$ contains the quantity $(\psi')^2$, which we view as $\psi' \psi'$. We substitute the right-hand side of 
equation~\eqref{eq:psiHamilt} for one factor $\psi'$ of this square
and leave the other  factor $\psi'$ in the resulting integrand.  
After making the change of variables $s=\psi(t)$, we obtain the equivalent relation
\begin{equation}
\label{eq:CondLaE2}
\frac{1}{\sqrt{\la}}\int_1^R \frac{\sqrt{\mu+u(s)}}{s}\,ds=A.
\end{equation}

We claim that  there exists a unique solution $(\mu,\la)$ of the
equations~\eqref{eq:CondLaE1} and~\eqref{eq:CondLaE2}. To prove this, 
 substitute for  $1/\sqrt{\la}$ from
equation~\eqref{eq:CondLaE1} into equation~\eqref{eq:CondLaE2} to obtain the
equation
\begin{eqnarray}
\label{eq:ProdInt}
A=f(\mu)&:=&\int_1^R \frac{\sqrt{\mu+u(s)}}{s} \,ds
\int_1^R \frac{1}{s\sqrt{\mu+u(s)}} \,ds.
\end{eqnarray}

Make the change of  variables $t=u(s)$ in both integrals in display~\eqref{eq:ProdInt} and then write $f(\mu)$ as a double integral to
obtain the formula
\begin{equation}
\label{eq:DoubleInt}
 f(\mu)=\int_0^{u(R)}\int_0^{u(R)}\frac{\sqrt{\mu+t}}{\sqrt{\mu+s}}\frac{1}{H(t)
H(s)}ds\,dt,
\end{equation}
where $H(t):=u^{-1}(t)u'(u^{-1}(t))\geq 0$ for all $t \in [0,u(R)]$.

By inspection of equation~\eqref{eq:ProdInt}, it is easy to see that $\lim_{\mu\to
0+}f(\mu)=+\infty$ and $\lim_{\mu \to \infty}f(\mu)=\log^2(R)$. We will show that $f$
is a decreasing function, which guarantees the existence of a unique solution of
the equation $f(\mu)=A$ for all $A>\log^2(R)$.

Using formula~\eqref{eq:DoubleInt}, we compute
$$
f'(\mu)=\frac{1}{2}\int_0^{u(R)}\int_0^{u(R)}\frac{s-t}{(\mu+s)^{3/2}(\mu+t)^{1/2}}\frac{1}{H(t) H(s)}ds\,dt.
$$
Let $D_+=\{(s,t)\in[0,u(R)]^2:s>t\}$ and $D_-=\{(s,t)\in[0,u(R)]^2:s<t\}$. After making a change of variables $\gamma(s,t)=(t,s)$, we see that
\begin{eqnarray*}
\lefteqn{{\int\!\!\int}_{D_+}\frac{s-t}{(\mu+s)^{3/2}(\mu+t)^{1/2}}\frac{1}{H(t) H(s)}ds\,dt=}\hspace{1in}\\
&&{\int\!\!\int}_{D_-}\frac{t-s}{(\mu+t)^{3/2}(\mu+s)^{1/2}}\frac{1}{H(t) H(s)}ds\,dt.
\end{eqnarray*}

Therefore,
\begin{eqnarray*}
2f'(\mu)&=&{\int\!\!\int}_{D_+\cup D_-}\frac{s-t}{(\mu+s)^{3/2}(\mu+t)^{1/2}}\frac{1}{H(t) H(s)}ds\,dt\\
\\
&=&-
{\int\!\!\int}_{D_-}\frac{(s-t)^2}{(\mu+s)^{3/2}(\mu+t)^{3/2}}\frac{1}{H(t) H(s)}ds\,dt<0.
\end{eqnarray*}
This completes the  proof that $f$ is a decreasing function.

There exists a unique solution $\mu$ of the equation $f(\mu)=A$ provided that $A>\log^2(R)$. The constant $\la$ is then easily found from equation~\eqref{eq:CondLaE1}.

Having found the unique solution $(\mu,\la)$ of the system~\eqref{eq:CondLaE1} and~\eqref{eq:CondLaE2}, we solve the initial 
value problem~\eqref{eq:psiIVP}. 
In fact, this initial value 
problem is equivalent to the integral equation
\begin{equation}
\label{eq:DEpsiEquiv}
\int_1^{\psi(t)}\frac{ds}{s\sqrt{\mu+u(s)}}=\frac{1}{\sqrt{\la}}t.
\end{equation}
It follows that the unique solution
$\psi$ of the
initial value problem~\eqref{eq:psiIVP} exists for all $t\in[0,1]$ and,
because of condition~\eqref{eq:CondLaE1}, satisfies
$\psi(1)=R$.
\end{proof}

Figs.~\ref{fig2} and~\ref{fig3} depict  graphs of the minimizer $\psi$ of the optimization
problem~\eqref{OptPr1} with $R=2$ and $A=f(\mu)$ in case $\mu=0.001$ for 
Fig.~\ref{fig2} and $\mu=500$ for 
Fig.~\ref{fig3}. Because $f$ is a decreasing function of $\mu$,
Fig.~\ref{fig2} corresponds to a larger constant $A$. These plots illustrate 
that second derivative of the radius function $\psi$ corresponding to 
the minimal morph increases as the constant $A$ in definition~\eqref{eq:GFunct} increases.

\begin{figure}
\begin{center}
\includegraphics[scale=1]{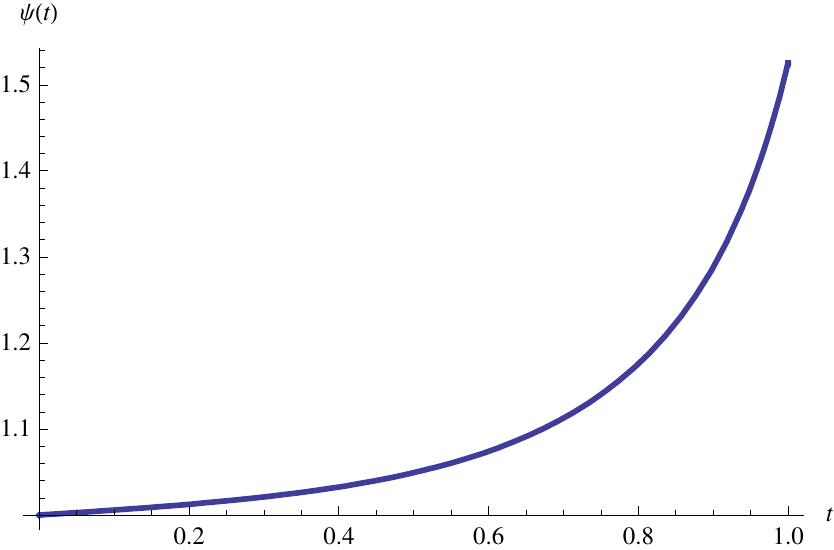}
\caption{Graph of the radius function $\psi$ with $R=2$, $\mu=0.001$, $\la=0.306067$, and $A=1.56296$.}
\label{fig2}
\end{center}
\end{figure}

\begin{figure}
\begin{center}
\includegraphics[scale=1]{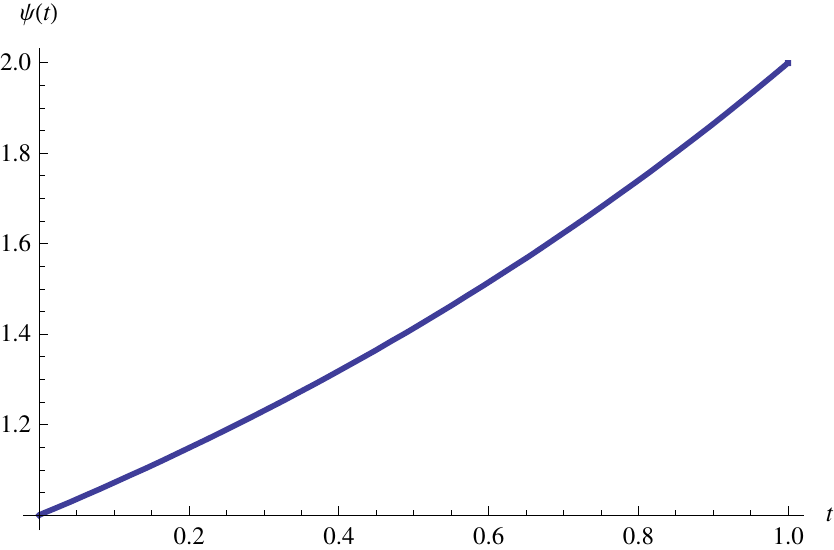}
\caption{Graph of the radius function $\psi$ with $R=2$, $\mu=500$, $\la=1045.58$, and $A=0.480456$.}
\label{fig3}
\end{center}
\end{figure}

\section{ Minimal Deformation Bending\\ of Two-Dimensional Spheres}
\label{Ch2:2dimSphere}
In this section we minimize the deformation energy functional $\Phi$ defined in
display~\eqref{eq:Funct1} under the assumptions
 $M=\Sph^2$ and $N=h_R(\Sph^2)=:R\,\Sph^2$ for some $R>0$, 
where $\Sph^2$ is the unit $2$-dimensional sphere in $\R^3$,
and $h_R$ is the radial map defined by $h_R(y)=R\,y$ for all $y \in \R^3$. 
As usual, the manifolds $M$ and $N$ are equipped with Riemannian
metrics
$g_M$ and $g_N$ respectively induced by the Euclidean metric $dx_1^2+dx_2^2+dx_3^2$ of $\R^3$.
The manifolds $M$ and $N$ are Riemann surfaces (see~\cite{Olson,Jost}).
We parametrize the spheres $\Sph^2$ and $R\, \Sph^2$ on the extended complex plane $\hat{\C}=\C \cup \infty$ using stereographic
projections. For $(y_1, y_2, y_3) \in \Sph^2$, the stereographic projection is 
given by
the expression $\pi(y_1,y_2,y_3)=\frac{y_1+ i y_2}{1-y_3}$. 
 We will show that maps of the form $h=f \circ h_R|_M$, 
where  $f$ is an isometry on $N$,
 minimize the functional $\Phi$ defined by equation~\eqref{eq:Funct1} in the class of all holomorphic diffeomorphisms 
from  $M$ to $N$.

We note that the holomorphic minimizers are critical points of the functional $\Phi$ on its 
natural domain $\diff(M,N)$ (see appendix~\ref{App:PhiEL} and corollary~\ref{Cor:RadialCrPt}).  

The parametrization $\phi:\hat{\C}\to\Sph^{2}$ is given by 
\begin{equation}
\label{phiS2}
\phi(u+ i v)=\bigg(\frac{2 u}{1+u^2+v^2}, \frac{2 v}{1+u^2+v^2}, 
\frac{-1+u^2+v^2}{1+u^2+v^2}\bigg)^T,
\end{equation}
and the 
parametrization $\phi_{R}:\hat{\C}\to R\,\Sph^{2}$ of $R\,\Sph^{2}$ is given by $\phi_{R}(u+i v)=R \phi(u+ i v)$.

In these coordinates, the Riemannian metrics $g_M$ and $g_N$ are defined by  
\begin{equation}
g_M(z, \bar{z})=\frac{4}{(1+|z|^2)^2} dzd\bar{z}
\end{equation}
and 
\begin{equation}
g_N(z, \bar{z})=\frac{4 R^2}{(1 +|z|^2)^2} dzd\bar{z}.
\end{equation}

Let $h \in \diff(M,N)$ be a holomorphic map. The local representation $(\phi_R)^{-1}\circ h\circ \phi:\hat{\C}\to\hat{\C}$ 
of $h$, which (by an abuse of notation) we shall denote by the same letter, is a holomorphic diffeomorphism of the extended complex plane onto itself. We
conclude that $h(z)$ has
the form $h(z)= M(z)$, where $M(z)=\frac{a z + b}{c z + d}$
 is a M\"{o}bius transformation and $a,b,c,d \in \C$ are such that $ad-bc \neq 0$. For such an $h$, it is easy to derive the formula 
\begin{equation}
h^\ast g_N(z,\bar{z})=\frac{4R^2 |bc-ad|^2}{\big(|a z +
b|^2+|c z +d|^2\big)^2}dzd\bar{z}.
\end{equation}
Hence, the problem of minimization of the deformation energy functional $\Phi$ defined in display~\eqref{eq:Funct1}
over all holomorphic diffeomorphisms from $\Sph^2$ to $R\, \Sph^2$
reduces to the problem of minimization of the function
\begin{equation}
\Psi(a,b,c,d)=\int_{\R^2}\Big(
\frac{R^2 |bc-ad|^2}
{\big(|a z +b|^2+|c z +d|^2\big)^2}-\frac{1}{(1+|z|^2)^2}\Big)^2(1+|z|^2)^2 \,du dv,
\end{equation}
where $z=u+i v$, over the 
 group $\Aut(\hat{\C})=\PGL(2,\C)$.
Recall that the elements of the projective general linear group $\PGL(2,\C)$ are the equivalence classes $[a,b,c,d]$, where $ad-bc\neq 0$ and
$(a',b',c',d')\in [a,b,c,d]$ if  $(a',b',c',d')=\lambda(a,b,c,d)$ for some $\lambda \in \C \backslash\{0\}$.

Recall that the group of all isometries of the Riemann sphere is the projective unitary group $\PU(2,\C)$; that is, every isometry $f$ of
$(\Sph^2,g_M)$ has the local representation  (via stereographic projection) 
$$f(z)=\frac{az-\bar{c}}{cz+\bar{a}},$$
where $a,c\in \C$ are such that $|a|^2+|c|^2=1$.

The functional $\Phi$ is invariant with respect to left compositions with isometries; that is, $\Phi(f \circ h)=\Phi(h)$
for every isometry $f \in \diff(N)$ and $h \in \diff(M,N)$.
Therefore, the reduced function $\Psi$ is well-defined on the quotient of 
$\PGL(2,\C)$ by $\PU(2,\C)$, which 
is the set of all equivalence classes
\begin{equation}
\bigg[\bigg[
\begin{array}{cc}
\al & \be\\
\gamma& \delta
\end{array}
\bigg]\bigg]=
\Bigg\{
\bigg(
\begin{array}{cc}
a & -\bar{c}\\
c& \bar{a}
\end{array}
\bigg)
\bigg(
\begin{array}{cc}
\al & \be\\
\gamma& \delta
\end{array}
\bigg): \bigg(
\begin{array}{cc}
a & -\bar{c}\\
c& \bar{a}
\end{array}
\bigg) \in \PU(2,\C)
\Bigg\}.
\end{equation} 

We note that the equivalence class 
$$
\bigg[\bigg[
\begin{array}{cc}
1 & 0\\
0& 1
\end{array}
\bigg]\bigg]
$$
consists of all the isometries of the unit sphere $(\Sph^2,g_M)$.

\noindent {\bf Proof of theorem~\ref{RiemSph}}

Statement (ii) of theorem~\ref{RiemSph}
follows immediately from Hurwitz's automorphisms theorem: \emph{The group of automorphisms of a compact
Riemann surface of genus greater than one is finite} (see~\cite{Olson}).

Statement (i) of theorem~\ref{RiemSph} is equivalent to the following result.
\begin{theorem}
\label{th:ShpMin}
The equivalence class of the
isometries of $(\Sph^2,g_M)$ is the unique minimizer of
the function $\Psi$  defined on the homogeneous space\\
$\PGL(2,\C)/\PU(2,\C)$; that is,
\begin{equation}
\Psi\Bigg(
\bigg[\bigg[
\begin{array}{cc}
1 & 0\\
0& 1
\end{array}
\bigg]\bigg]
\Bigg)
\leq \Psi\Bigg(
\bigg[\bigg[
\begin{array}{cc}
a & b\\
c& d
\end{array}
\bigg]\bigg]
\Bigg)
\end{equation}
for all $$\bigg[\bigg[
\begin{array}{cc}
a & b\\
c& d
\end{array}
\bigg]\bigg] \in \PGL(2,\C)/\PU(2,\C).$$
\end{theorem}
\begin{proof}

The function $\Psi$ is well-defined on the homogeneous space
$\PGL(2)/\PU(2)$. Thus, all values of 
$\Psi$ are obtained by choosing its domain to consist of one representative from each equivalence class. 

We claim that each equivalence class
\begin{equation*}
\bigg[\bigg[
\begin{array}{cc}
\al & \be\\
\gamma& \delta
\end{array}
\bigg]\bigg]
\end{equation*}
has a representative of the form 
\begin{equation*}
\bigg(
\begin{array}{cc}
1 & 0\\
z& r
\end{array}
\bigg),
\end{equation*} 
for some  $z \in \C$ and $r \in \R_+$. 

To prove the claim, note that (without loss of generality) we may assume the determinant of the given representative is unity; that is, $\alpha \delta - \beta \gamma=1$. We wish to prove the existence of  $a,c \in \C$ so that 
\begin{equation}
\label{ss4}
\bigg(
\begin{array}{cc}
a & -\bar{c}\\
c& \bar{a}
\end{array}
\bigg)
\bigg(
\begin{array}{cc}
\al & \be\\
\gamma& \delta
\end{array}
\bigg)=
\bigg(
\begin{array}{cc}
1 & 0\\
z& r
\end{array}
\bigg)
\end{equation}
for some $z \in \C$ and $r \in \R_+$.
In other words, it suffices to solve the system of linear equations 
\begin{equation}
\left\{
\begin{array}{l}
a \al-\bar{c} \gamma=1,\\
a \be -\bar{c} \delta=0.
\end{array}
\right.
\end{equation}
In view of the equation $\alpha \delta - \beta \gamma=1$, it follows that $a=\delta$ and $c=\bar{\beta}$.
By substitution of $a$ and $c$ into equation~\eqref{ss4}, we find that
$z=\bar{\beta} \alpha+\bar{\delta}\gamma$ and $r=|\beta|^2+|\delta|^2$. This proves the claim. 

By the claim,  it suffices to consider the value of $\Psi$ only at points of the
form $(1,0,q e^{i \psi}, r)$, where $q\in \R$, $r \in \R_+$, and $\psi \in [0,2 \pi)$. Thus, the theorem is an immediate consequence of the following proposition. 

\noindent\emph{The function $\bar{\Psi}:\R \times[0,2\pi]\times \R_+\to \R$ given by
\begin{equation}
\label{nnn11}
\bar{\Psi}(q, \psi, r)=\Psi(1,0,q e^{i \psi}, r)
\end{equation} 
attains its global minimum on the set of points $(0,\psi,1)$.} 

To prove this result, let us first calculate the integral that represents the 
function $\bar{\Psi}$.

After passing to polar coordinates ($u=\rho \cos \phi$ and $v=\rho \sin \phi$), 
we represent $\bar{\Psi}$ in the form
\begin{equation*}
\bar{\Psi}(q,\psi,r)=\int_0^{\infty}\!\!\! \int_{0}^{2 \pi} \Big[
\frac{R^2 r^2}{(\xi +\eta \cos(\phi+\psi))^2}-\frac{1}{(1+\rho^2)^2}
\Big]^2(1+\rho^2)^2\rho\,d \phi d\rho,
\end{equation*}
where $\xi=\rho^2+\rho^2q^2+r^2$ and $\eta=2\rho qr$.
Since the integrand is periodic with respect to $\phi$ and we are integrating over one period, $\bar{\Psi}(q,\psi,r)$ does not 
depend on $\psi$; that is,
\begin{eqnarray*}
\bar{\Psi}(q,\psi,r)=\int_0^{\infty}\!\!\! \int_{0}^{2 \pi} \Big[
\frac{R^2 r^2}{(\xi +\eta \cos\phi)^2}-\frac{1}{(1+\rho^2)^2}
\Big]^2(1+\rho^2)^2\rho\,d \phi d\rho.
\end{eqnarray*}
The inner integral of the equivalent iterated integral is
\begin{eqnarray}
\nonumber
K(\rho):&=&\int_0^{2\pi}
\Big[
\frac{R^2 r^2}{(\xi +\eta \cos(\phi))^2}-\frac{1}{(1+\rho^2)^2}
\Big]^2(1+\rho^2)^2\rho\, d\phi\\
\nonumber
&=&\Big[R^4r^4\int_0^{2\pi}\frac{1}{(\xi +\eta \cos(\phi))^4}\,d\phi\\
\nonumber
&&{}- 2R^2r^2\frac{1}{(1+\rho^2)^2}
\int_0^{2\pi}\frac{1}{(\xi +\eta \cos(\phi))^2}\,d\phi\\
\label{Krho}
&&{}+ 2\pi\frac{1}{(1+\rho^2)^4}\Big]\rho(1+\rho^2)^2.
\end{eqnarray}
Taking into account the inequalities $\xi>|\eta|$ and
$\eta>0$, the integrals in the previous expression are elementary; their values  are given by
$$
\int_0^{2\pi}\frac{1}{(\xi +\eta \cos(\phi))^2}\,d\phi=\frac{2 
\pi \xi}{(\xi^2-\eta^2)^{3/2}}
$$
and
$$
\int_0^{2\pi}\frac{1}{(\xi +\eta \cos(\phi))^4}\,d\phi=
\frac{\pi\xi(2 \xi^2+ 3 \eta^2)
}{(\xi^2-\eta^2)^{7/2}}.
$$
By substitution into equation~\eqref{Krho}, we find
that
\begin{eqnarray}
\label{PsiExplicit}
\bar{\Psi}(q,\psi, r)&=& \int_0^{\infty}K(\rho)\, d\rho\\
\nonumber
&=&\pi-2\pi R^2+\frac{\pi R^4}{3r^2}\big(1+q^2+(r-1)r\big)(1+q^2+r+r^2).
\end{eqnarray}

The minimum of the function $F(q,r)=\bar{\Psi}(q,\psi,r)$ on $\R \times \R_+$, is easily determined. Indeed, $(0,1)$ is the only critical point of $F$. Also, the Hessian of $F$ is 
\[
D^2F(q,r)
=
\left(
\begin{array}{cc}
\frac{4\pi R^4(1+3 q^2+r^2)}{3 r^2} & -\frac{8\pi  R^4 q(1+q^2)}{3 r^3}\\[0.05in]
-\frac{8\pi  R^4 q(1+q^2)}{3 r^3} & 
\frac{2\pi R^4(3+6 q^2+3 q^4+r^4)}{3 r^4}
\end{array}
\right).
\]
We note that $\pd^2 F(q,r)/\pd q^2$  and the determinant of the Hessian
\[
\det\big(D^2F(q,r)\big)=\frac{8\pi^2 R^8}{9 r^6}(3+q^6+3 r^2+r^4+r^6+q^4(5+3 r^2)+q^2(7+6 r^2+3 r^4))
\]
are both positive by inspection. By Sylvester's criterion, the Hessian is positive definite over the entire domain of $F$; therefore, $F$ is convex. If follows that $(0,1)$ is the unique global minimizer of $F$. The minimum of $F$ is 
$$F(0,1)=\pi(R^2-1)^2.$$

Clearly, points of the form  $(0,\psi,1)$ are the global minima of $\bar{\Psi}$ on its domain
$\R \times [0,2\pi] \times \R_+$.
\end{proof}

\section{Discussion}

For diffeomorphic hypersurfaces $M$ and $N$ in a Euclidean space, we have defined functionals that measure how well a
diffeomorphism $\psi: M \to N$ preserves the geometry of $M$ and proved that
minimizers of these functionals exist.  Since our functionals involve
comparisons of the first and the second fundamental forms on $M$ with 
the pull-backs of the corresponding fundamental forms from $N$ to $M$ by $\psi$,
they measure the extent to which $\psi$ changes the size of vectors carried 
from $M$ to $N$ and the extent to which it preserves the amount of bending of 
the unit normal vector field on $M$. In addition, to maintain flexibility for applications (where some particular combination of the measurements given by the first and the second fundamental forms is desired) we allow the measurement of distortion length and distortion bending to be weighted. 

Even in the case where $M$ and $N$ are one-dimensional, there are examples of such manifolds where the distortion energy functional for the first fundamental form has no minimizer in $\diff(M,N)$ (see~\cite{OpusculaPaper}).  This fact and the complexity of the corresponding Euler-Lagrange equation (derived in appendix~\ref{App:PhiEL}) evidence the difficulty of the general problem of minimization of distortion energy functionals over $\diff(M,N)$.

One of the main ideas used successfully here is to restrict, via an appropriate linearization, the space of diffeomorphisms on which the distortion energy functionals are minimized. 
More precisely, instead of working directly with maps in  $\diff(M,N)$,  we consider time-dependent
vector fields on the ambient Euclidean space that generate the desired diffeomorphisms.  Each such vector field determines an evolution family that can be applied to $M$, for values of the evolution parameter in the interval $[0,1]$,  to produce a morph carrying $M$ to
some diffeomorphic end-hypersurface. We restrict our attention to the class of 
vector fields for which this end-hypersurface is  $N$. The advantage is clear: The elements of $\diff(M,N)$ are replaced by elements of a vector space.  Using this approach,
which already appears in the literature on 
image deformation (see ~\cite{Cao, Dupuis, GrenMill, TYMeta}), we are able to prove the existence of
minimizers for the distortion morphing functional $\SE$ and the distortion bending functional $E$.   The functional $\SE$  is an appropriate choice to measure distortion in case we wish to include the deviation of the geometry of the intermediate surfaces from the original surface $M$.  On the other hand, the distortion functional $E$ is appropriate if we wish to ignore the distortions represented by the intermediate surfaces and only consider the distortion caused by mapping $M$ to $N$. 

We have also gone beyond the existence of minimizers. Indeed, we have  determined minimizers of our distortion energy functionals for some classes of one-dimensional manifolds and for Riemann surfaces. While these special cases have independent interest, they also serve to test conjectures. 

The existence of minimizers provides a hunting license for seeking approximations of these minimal distortion diffeomorphisms and morphs via numerical methods. The known minimizers for special cases can be used to test numerical algorithms.

Our admissible set $\SA_P^k$ of time-dependent vector fields is the
closed ball of radius $P>0$ in the Hilbert space $\SH^k$
of vector fields intersected with the set $\SB$ of all vector fields that carry $M$ to $N$. The  boundedness of $\SA_P^k$ in $\SH^k$ is an essential ingredient in our  proof of the existence of distortion energy minimizers. Thus, there is a natural (open) question: Is there a minimizer in the interior of $\SA_P^k$ 
(in the relative topology of the $\SB$) or are all minimizers on the boundary of this set?

We mention some evidence that suggests there are energy minimizers of the functional $E$ in the interior of $\SA_P^k$ for 
$P>0$ sufficiently large. We expect the same result for $\SE$, but this case seems to be much more subtle.

Our result, discussed in section~\ref{sec:3.1}, that (1) the radial map is a minimal distortion map in $\diff(\Sph^1,\Sph^1_R)$ in case $R>1$  and (2) there is a time-dependent vector field whose time-one map is the radial map, shows that (at least for this special case) interior minimizers exist. Indeed, for $P$ larger than the Hilbert-norm of our time-dependent vector field that produces the radial map, this vector field is in the interior of the admissible set $\mathcal{A}^k_P$.  

To prove the existence of an interior minimizer in general, it suffices to show that there is a number $P_0>0$ such that for every time-dependent vector field $v \in \SH^k$ that carries  $M$ to  $N$  such that its Hilbert-norm $\|v\|_{\SH^k}\ge P_0$, there exists a vector field $v'$ whose Hilbert-norm is less than $P_0$ and whose energy (as measured by $E$ or $\SE$) does not exceed the corresponding energy of $v$.

Suppose that for some large $P$ we have a minimizer of the functional $E$, a time-dependent vector field $v$ whose Hilbert-norm is $P$; that is, $v$ lies on the boundary of the admissible set. The time-one map of $v$ determines the value of $E$. We seek a new vector field $v'$ with smaller Hilbert-norm and the same time-one map as $v$; it would be our desired interior minimizer. A large Hilbert-norm for $v$ suggests that its integral curves have large lengths or largeness in some measure of bending that takes into account the space derivatives of $v$ which are used to compute the Sobolev norm. By shortening and straightening the integral curves of $v$ in some subset of the space $\Omega\times [0,1]$, where our time-dependent vector fields are defined, we can construct a new time-dependent vector field $v'$ whose Hilbert-norm is strictly smaller than the norm of $v$ and whose time-one map is the same as the time-one map of $v$.  

Our choice of admissible time-dependent vector fields, which are defined on 
$\Omega\times [0,1]$, for $\Omega$ a ball  in  $\R^{n+1}$ containing the manifolds 
$M$ and $N$, can be replaced by other sets of functions chosen to not contain 
extraneous information. For example, it might be desirable to consider only 
the curves leading from $M$ to $N$ and not all the integral curves generated 
by our time-dependent vector fields defined on $\Omega\times [0,1]$. One 
alternative admissible set is a subset of a  class of functions we call 
\emph{development vector fields}. They are defined to be functions 
$v:M\times [0,1]\to T\R^{n+1}$ (in an appropriate function space) that generate morphs $F^v:M\times [0,1]\to \R^{n+1}$ via the following construction. Let $\tau_p^q : T_p\R^{n+1}\to T_q \R^{n+1}$ be the parallel transport from the tangent space at $p$ to the tangent space at $q$ defined by the Euclidean metric (or perhaps some other metric) on  $\R^{n+1}$. For each $p\in M\subset \R^{n+1}$, the solution of the initial value problem
\[ \dot c_p=\tau_p^{c_p}v(p,t), \qquad c_p(0)=p\]
is a curve in $\R^{n+1}$ starting on $M$ at $p$. Thus, we have defined a morph $F^v:M\times [0,1]\to \R^{n+1}$ given by $F^v(p,t):=c_p(t)$. The new admissible set is all  development vector fields that are in some closed ball of radius $P>0$ of the corresponding Hilbert space such that $F^v(M,1)=N$. 

We also mention the possibility of defining a {new} norm on our time-dependent or development vector fields by reversing the order of integration:
\[ 
\|v\|^2:=\int_\mathcal{R} \int_0^1\sum_{|\alpha|\le k}|D^\alpha_x v(x,t)|^2 \,dt dx,
\] 
where $\mathcal{R}$ is $\Omega$ for time-dependent vector fields and $M$ 
for development vector fields. This norm may differ from the norm $\|v\|_{\SH^k}$
used throughout this paper as functions in the Hilbert space $\SH^k$  are not necessarily
 $L^1$ in the joint $(x,t)$ variable on $\Omega\times[0,1]$.
Possible advantages of this approach are a better understanding of the relation between this norm and the shapes of the (integral) curves that are used to define morphs and, by first integrating over the temporal parameter, the Sobolev-norm can be viewed as a norm for functions defined on the compact manifold $M$.  Consideration of this norm led us to the concept of a  development vector field. Alternative approaches to the problem of minimal morphing using these ideas are a subject for future research.

\appendix
\section{The Euler-Lagrange Equation for the\\ Deformation Energy Functional}
\label{App:PhiEL}

We will determine the  Euler-Lagrange equation for the 
deformation energy functional $\Phi:\diff(M,N)\to\R_+$ defined in display~\eqref{eq:Funct1}. 
Let $c:(-\ep,\ep)\to \Diff(M,N)$ be a $C^1$ curve at $h \in \diff(M,N)$, which 
we call a \emph{variation} of $h$. The equivalence class 
$[c]_h \in T_h\diff(M,N)$ can be identified with the smooth vector field 
$Y \in \Chi(N)$ defined by $Y(q)=\frac{d}{dt}c(t)\circ h^{-1}(q)$ for all 
$q \in N$. We call the vector field $Y \in \Chi(N)$ a \emph{variational vector field} of $h \in \diff(M,N)$.
We will compute the first variation $D\Phi(h)Y$ for all directions $Y \in \Chi(N)$.

Consider a smooth vector field $X \in \Chi(M)$ with flow $\phi_t$ and  a diffeomorphism $h \in \diff(M,N)$, and suppose that the variation $c(t)=h\circ\phi_t$ induces the variational vector field $Y=h_\ast X \in \Chi(N)$.
The diffeomorphism $h$ is a critical point of the functional $\Phi$ if and only if
\begin{equation}
\label{EL3}
\frac{d}{dt} \Phi(h \circ \phi_t)|_{t=0 }=
D\Phi(h) h_\ast X=2\int_M G(h^\ast g_N-g_M, L_X h^\ast g_N)\w_M=0
\end{equation}
for all $X\in \Chi(M)$, where $G=g_M^\ast\otimes g_M^\ast$.
 
Let $\na$ and $\bar{\na}$ be Riemannian connections on $M$ compatible with the Riemannian metrics $\al=g_M$ and $\be=h^\ast g_N$ respectively, and denote the corresponding Christoffel symbols of $\na$ and $\bar{\na}$ by $\Gamma^i_{jk}$ and $\bar{\Gamma}^i_{jk}$.

Let $Y$ be a smooth vector field on $M$ expressed in components by 
$Y=Y^k \frac{\pd}{\pd x^k}$. The components of the Lie derivative of the Riemannian metric $\be$ in the direction of the vector field $Y$ are
\begin{equation}
\label{lienabla1}
[L_Y \beta]_{km}=\bar{\na}_k Y_m + \bar{\na}_m Y_k,
\end{equation}
where $Y_m=[\be]_{mj} Y^j$ are the lowered coordinates of $Y$ via the Riemannian metric $\be$ (see~\cite{AMR}). 

Recall that $\ST^r{}_s(M)$ is the set of all continuous tensor fields on $M$ 
contravariant of order $r$ and covariant of order $s$, or type $(r,s)$.

\begin{definition}
\label{defB}
(i) Define the tensor field $B=(\beta-\alpha)^{\#\#} \in \ST^2{}_0(M)$. 
In other words, $B$ equals the strain tensor field  
$h^\ast g_N- g_M
$
with both indices raised via the Riemannian metric $\al=g_M$. Its components are given by
$B^{km}=(\be_{ij}-\al_{ij})\al^{ik}\al^{jm}$.

(ii) The bilinear form $A:\Chi(M)\times\Chi(M)\to \Chi(M)$ 
is defined by
\begin{equation}
\label{SS}
A(X,Y)=\bar{\na}_X Y-\na_X Y
\end{equation}
for all $X,Y \in \Chi(M)$.
\end{definition}

\begin{remark}
The bilinear form $A$ can be viewed as a tensor field of type~$(1,2)$ on $M$ with components
\begin{equation}
\label{Scomp}
A^m{}_{kp}=\bar{\Gamma}_{kp}^m -\Gamma_{kp}^m
\end{equation}
(see ~\cite{KN}, proposition~7.10).
\end{remark}

Recall that the divergence of a tensor field $\tau\in\ST^r{}_s(M)$ is defined to be (see \cite{H})
\begin{equation}
\label{tensordiv}
\div \tau=C_{s+1}^r(\na \tau),
\end{equation}
where $C_i^j$ denotes the contraction in lower $i$ and upper $j$ index.
The divergence of $\tau$, $\div \tau$, is a tensor of type~$(r-1,s)$.

For two tensor fields $\theta \in \ST^r{}_2(M)$ and $\tau \in \ST^2{}_s(M)$, 
$\theta:\tau$ denotes the type $(r,s)$ tensor field obtained by the contraction of the two covariant degrees of $\theta$ with the two contravariant degrees of $\tau$.

\begin{proposition}
\label{fundlemma1}
Let $M$ be a compact, connected, and oriented smooth Riemannian $n$-manifold 
without boundary isometrically embedded into $\R^{n+1}$.
For the functional $\Phi(h)=\int_M 
\|h^\ast g_N-g_M\|^2\,\w_M$ with domain $\diff(M,N)$, we have that 
\begin{equation}
\label{DPhiY1}
D\Phi(h)(h_\ast Y)=-4\int_M g_M(\div B+ A:B,Y)\,\w_M
\end{equation}
for all vector fields $Y \in \Chi(M)$, where the tensors $A$ and $B$ are 
as in definition~\ref{defB}.
Moreover, 
$h \in \diff(M,N)$ is a critical point of the functional $\Phi$ if and only if 
$$\div B+A:B=0.
$$
The latter equation can be rewritten in components as follows:
\begin{equation}
\label{ELComp}
\pd_k B^{km}+\Gamma_{kp}^p B^{km}+\bar{\Gamma}_{kp}^m B^{kp}=0
\end{equation}
for all $m=1,2,\ldots,n$.
\end{proposition}
\begin{proof}
For given $Y \in \Chi(M)$, consider the vector field 
\begin{equation}
X=(\beta_{ij}-\alpha_{ij})\al^{ik}\al^{jm}Y_m \frac{\pd}{\pd x_k},
\end{equation}
where $\al=g_M$ and $\be=h^\ast g_N$.
Although we describe $X$ pointwise using local coordinates, $X$ is a well defined smooth 
vector field on $M$ because it is obtained by various contractions of the tensor fields 
$\alpha, \beta$, and $ Y$. It can be verified that the divergence of the vector field $X$ 
with respect to the  Riemannian metric $g_M$ is
$$
\div _{g_M} X= \frac{1}{2}G(\be-\al, L_Y \be) + g_M(\div B+A:B,Y).
$$
Because $M$ is without boundary, $\int_M \div _{g_M} X \w_M=0$. Using this and the equality 
$$
D\Phi(h)(h_\ast Y)= 2\int_M G(\beta-\alpha,L_Y \be)\w_M,
$$
we conclude that 
$$
D\Phi(h)(h_\ast Y)=-4\int_M g_M(\div B+ A:B,Y)\w_M
$$
as required.
\end{proof}

\begin{corollary}
\label{Cor:RadialCrPt}
Let $M$ be a manifold as in proposition~\ref{fundlemma1} and $h\in \diff(M,N)$. If 
$h^\ast g_N=R^2 g_M$ for some $R \in \R$, then $h$ is a critical point of 
the functional $\Phi$. 

In particular,
 let $h_R:\R^{n+1}\to\R^{n+1}$ be the radial map defined by $h_R(p)=Rp$ for all $p \in \R^{n+1}$, where $R>0$. Assume that $N=h_R(M)$ is a rescaled
version of the manifold $M$, and the Riemannian metrics $g_M$ and $g_N$ on the manifolds $M$ and $N$ are inherited from $\R^{n+1}$. Then every composition $h=f \circ h_R|_M$ 
of the radial map $h_R|_M$ with an isometry $f \in \diff(N)$ is a critical
point of the functional $\Phi$.
\end{corollary}
The proof of this corollary is an easy verification of equation \eqref{ELComp}.

\section{Existence and Convergence\\ Results for Evolution Operators}
\label{Append:A}

In this section we state results on existence and convergence of certain evolution operators.

We denote the Euclidean norm of an element $A \in \R^m$, where $m \in \N$, by $|A|$ and  the Hilbert space $L^2(0,1;V^k)$ by ${\SH^k}$, where the Sobolev space 
$V^k=W^{k,2}_0(\Omega;\R^{n+1})$ is embedded into $C^r(\bar{\Omega};\R^{n+1})$ and $r\geq 2$. Recall that Sobolev's theorem guarantees the latter embedding if
$k\geq (n+1)/2+r+1$. The following lemma is proved in~\cite{Dupuis}.

\begin{lemma}[Dupuis, Grenander, Miller]
\label{Lemma:exist}
For every time-dependent vector field $v \in {\SH^k}$ and $t_0 \in [0,1]$,  there exists a function $\phi:[0,1]\times \Omega \to \R^{n+1}$ such that 
$t\mapsto \phi(t,x)$ is the unique absolutely continuous solution of the initial value problem
\begin{equation}
\label{FF1}
\left\{
\begin{array}{ll}
\frac{dq}{dt}=v(q,t),\\
q(t_0)=x
\end{array}
\right.
\end{equation} 
for all $t \in [0,1]$. Moreover, the function $x \mapsto \phi(t,x)$ is a homeomorphism of $\Omega$. 
\end{lemma}

For every $v\in {\SH^k}$ and $x \in \Omega$, let $F^v(x,t)$ be the solution of the evolution equation ${dq}/{dt}=v(q,t)$ with the initial condition $F^v(x,0)=x$. For a  function $f \in C^r(\Omega)$, denote
$$
\|f\|_{r,\infty}=\sum_{\al,|\al|\leq r}\sup_{x \in \Omega}|D^\al f(x)|,
$$
where $\al=(\al_1,\ldots,\al_{n+1})$ is a multi-index with nonnegative integer components, $|\al|=\al_1+\ldots+\al_{n+1}$, and 
$D^\al f=\frac{\pd^{|\al|}f}{\pd x_1^{\al_1}\ldots x_{n+1}^{\al_{n+1}}}$.

More general versions of the following two lemmas are proved in~\cite[Appendix~C]{TrouveYounes}.

\begin{lemma}[Trouve, Younes]
\label{lemma:Diff}
If $v \in {\SH^k}$ and $F^v:\Omega\times[0,1]\to \R^{n}$ is defined as above,  then the function $x\mapsto F^v(x,t)$ is in class $C^r(\Omega)$
and, for all $ q \leq r$,
$$
\frac{\pd}{\pd t}D^q_xF^v(x,t)=D^q_x\big(v(F^v(x,t),t)\big),
$$
where $D^q_x$ denotes the derivative with respect to $x$ of order $q$.
Moreover, there exist positive constants $C$ and  $C'$ such that
$$
\sup_{t\in[0,1]}\|F^v(\cdot,t)\|_{r,\infty}\leq Ce^{C'\|v\|_{{\SH^k}}}
$$
for all $v\in {\SH^k}$.
\end{lemma}

Recall that we say $v^{l}\wto v$ weakly in ${\SH^k}$ as $l \to \infty$ if $\langle v^{l}-v,w\rangle\to 0$ as $l\to \infty$ for all $w \in {\SH^k}$.
\begin{lemma}[Trouve, Younes]
\label{lemma:DerBded}
If the sequence $\{v^l\}_{l=1}^\infty\subset {\SH^k}$ converges weakly to $v \in {\SH^k}$ as $l\to \infty$, 
 then 
$$
\sup_{t \in [0,1]}\|F^{v^l}(\cdot,t)- F^v(\cdot,t)\|_{r-1,\infty}\to 0
$$
as $l \to \infty$. 
Moreover, there exists a constant $K>0$ such that
\begin{equation}
\label{eq:DerBded}
 \sup_{t\in[0,1]}\|F^{v^l}(\cdot,t)\|_{r,\infty}\leq K
\end{equation}
for all $l \in \N$.
\end{lemma}

\section{Proof of Lemma~\ref{lemma:NormsCont}}\label{Append:B}

\noindent\textbf{Proof of lemma~\ref{lemma:NormsCont}.}
\begin{proof}
The continuity of  $\eta$ in the $\beta$ variable is an immediate consequence of
 the
 equivalence of all the norms on the finite-dimensional space $T^{0}{}_{s}(X)$.

Let us  show that $\eta$ is continuous in the $g$ variable.

Fix $\beta \in T^{0}{}_{s}(X)$ and $g \in \Metrics(X)$. 
Let $K(g)$ be a positive constant such that $|v|_{g_{X}}\leq K(g)$ for all $v
\in S_{g}$, and let $C(g)$ be a positive constant such that 
$|v|_{g_{X}}\leq C(g)|v|_g$ for all $v
\in X$.

Choose $\ep >0$. We leave it to the reader to show that there exists $\delta>0$ such that for all $v \in S_{g}$ and
$g' \in \Metrics (X)$ satisfying $\|g-g'\|_{g_{X}}<\delta$ we have
\begin{equation}
\label{eq:sphCompare}
|v-\frac{v}{|v|_{g'}}|_{g}<\frac{\ep}{s\|\beta\|_{g_{X}}C(g)(2K(g))^{s-1}}.
\end{equation}

  Let the vectors $v_{1}, \ldots, v_{s} \in S_{g}$ be such that
$|\beta(v_{1},\ldots,v_{s})|=\|\beta\|_{g}$, and define
 $u_{i}:={v_{i}}/{|v_{i}|_{g'}} \in S_{g'}$, where $\|g-g'\|_{g_{X}}<\delta$. Then
\begin{eqnarray*}
\|\beta\|_{g}&=&|\beta(v_{1},\ldots,v_{s})|\\
&=&|\beta(v_{1}-u_{1}+u_{1},\ldots,v_{s}-u_{s}+u_{s})|\\
&\leq&|\beta(u_{1},\ldots,u_{s})|+\|\beta\|_{g_{X}}C(g)(2K(g))^{s-1}\sum_{i=1}^{s}|u_{i}-v_{i}|_{g}\\
&<& \|\beta\|_{g'}+\ep.
\end{eqnarray*}

Therefore,
$\|\beta\|_{g'}> \|\beta\|_{g}-\ep$. The inequality $\|\beta\|_{g}> \|\beta\|_{g'}-\ep$ can be shown in the same
fashion. Hence, the function $\eta$ is continuous in the $g$
variable.
\end{proof}

\section{Proofs of lemmas~\ref{lm:psiPr2Ex1R} and~\ref{lm:PsiNecCond}}\label{Append:C}
\textbf{Proof of lemma~\ref{lm:psiPr2Ex1R}}
\begin{proof}
The proof consists of two main steps:
(1) Using the direct method of the calculus of variations, we will prove the existence of a minimizer for the auxiliary optimization problem
\begin{equation}
\label{OptPr2.5}
\begin{array}{lll}
\mbox{minimize }  J(\psi), \\
\psi\in Q^{1,4/3}:=\{\phi \in W^{1,\frac{4}{3}}(0,1): \psi(0)=1, \psi(1)=R\}\\
\mbox{subject to } G(\psi)\leq 0.
\end{array}
\end{equation}
(2) We will show that the minimizer for problem~\eqref{OptPr2.5}  is in 
$W^{1,2}(0,1)$.\\
If follows that this minimizer is a minimizer of the optimization problem~\eqref{OptPr2}.

Let $\{\psi_n\}_{n=1}^\infty\subset Q^{1,4/3}$ be a 
minimizing sequence for the optimization problem~\eqref{OptPr2.5}. In particular, $G(\psi_n)\leq 0$  for every positive integer $n$.  In symbols, 
$$
J(\psi_n)\to \inf_{\psi\in Q^{1,4/3}, G(\psi)\leq 0} J(\psi).
$$ 

We claim that the minimizing sequence is
bounded in $W^{1,4/3}(0,1)$. To prove this fact, we use 
the triangle inequality for the $L^2(0,1)$ norm to make the estimate
\begin{eqnarray*}
\{\int_0^1 \psi_n^4\,dt\}^{1/2}&=&\{\int_0^1(\psi_n^2-1+1)^2\,dt\}^{1/2}\\
&\leq&
J(\psi_n)^{1/2}+1\\
&\leq& \sqrt{M}+1,
\end{eqnarray*}
where $M>0$ is a uniform bound for the convergent sequence $\{J(\psi_n)\}_{n=1}^\infty$.
By H\"{o}lder's inequality with the
conjugate constants $3$ and $3/2$, 
\begin{eqnarray*}
\int_0^1 |\psi_n'|^{4/3}\,dt&=&\int_0^1|\psi_n|^{4/3} 
\Big(\frac{|\psi_n'|}{|\psi_n|}\Big)^{4/3}\,dt\\
&\leq& \Big\{\int_0^1|\psi_n|^{4}\,dt\Big\}^{1/3}
\Big\{\int_0^1\Big(\frac{|\psi_n'|}{|\psi_n|}\Big)^{2}\,dt\Big\}^{2/3}\\
&\leq& \big((\sqrt{M}+1)A\big)^{2/3},
\end{eqnarray*}
as required.

Because the Banach space $W^{1,4/3}(0,1)$ is reflexive, $\psi_n\wto\psi$ weakly in $W^{1,4/3}(0,1)$ for some $\psi \in
W^{1,4/3}(0,1)$, up to a subsequence. We have $\psi \in Q^{1,4/3}$ because the subspace
$W_0^{1,4/3}(0,1)$ is weakly closed in $W^{1,4/3}(0,1)$. 

The integrands $(\psi^2-1)^2+(\psi-1)^2$ and $(\psi')^2/\psi^2$ of $J$ and $G$
respectively are both convex functions of $\psi'$. Therefore, the functionals
$J$ and $G$ are weakly lower semicontinuous in $W^{1,4/3}(0,1)$ (see
theorem~1, Sec.~8.2 in \cite{Evans}). But then $G(\psi)\leq \liminf_{n \to
\infty}G(\psi_n)\leq 0$ and $\psi \in Q^{1,4/3}$ solves optimization
problem~\eqref{OptPr2.5}.

To prove that $\psi\geq 1$, let us assume, on the contrary, that there exists (in the usual topology of $[0,1]$) 
an open set $W$ of positive measure
such that $\psi(t)<1$ for all $t \in W$. Define the cut-off function $h_1 \in Q^{1,4/3}$
 by
$
h_1(t)=\max\{1,\psi(t)\}.
$
It is easy to check that $G(h_1)\leq 0$ and that $J(h_1)<J(\psi)$, which
contradicts the minimizing property of $\psi$.
The inequality $\psi(t)\leq R$ for all $t \in[0,1]$ can be verified in a 
similar fashion, using the cut-off function $h_2(t)=\min\{R,\psi(t)\}$.

Using the
inequality $\psi\leq R$, we have the estimate
$$
\int_0^1(\psi')^2\,dt=\int_0^1\psi^2\Big(\frac{\psi'}{\psi}\Big)^2\,dt\leq R^2
\,A.
$$
Therefore, $\psi$ belongs to the space $W^{1,2}(0,1)$.

Finally, the uniqueness of $\psi$ follows from the fact that the equality
$J(\psi_1)=J(\psi_2)$, where $\psi_1,\psi_2\in Q^{1,2}$ are such that
$1\leq\psi_{1,2}\leq R$, implies
$u\circ\psi_1(t)=u\circ\psi_2(t)$ for all $t \in [0,1]$, where the function
$u(s)=(s^2-1)^2+(s-1)^2$ is strictly increasing on $(1,+\infty)$.

\end{proof}

\textbf{Proof of lemma~\ref{lm:PsiNecCond}}
\begin{proof}
Statements~(i) and~(ii) follow from the
generalized Kuhn-Tucker theorem (see~theorem~1, Sec. 9.4 in~\cite{Luenberger}).

We will verify that the minimizer $\psi$ is a regular point of the inequality 
$G(\psi)\leq 0$. We leave it to the reader to verify that the functionals $J$ and
$G$ are Gateaux differentiable at $\psi\geq 1$.

It suffices to show that there exists $h\in W^{1,2}_0(0,1)$ such that
$$
\delta G(\psi,h)=\int_0^1\frac{\psi'}{\psi^3}(h'\psi-\psi' h)\,dt<0,
$$
where $\delta G(\psi,h)$ is the Gateaux derivative of $G$ in the direction $h$.

Assume, on the contrary, that 
\begin{equation}
\label{eq:psiRegPt}
\int_0^1\frac{\psi'}{\psi^3}(h'\psi-\psi' h)\,dt=0
\end{equation}
for all $h\in W^{1,2}_0(0,1)$. Then $\psi$ satisfies 
the Euler-Lagrange equation for the functional $G$ whose associated  Lagrangian
$(\psi'/\psi)^2$  has a positive second derivative with respect to $\psi'$.
By a regularity result for weak solutions of Euler-Lagrange equations (see theorem~1.2.3, Sec.~1.2 in~\cite{JostCV}), 
$\psi$ is of class $C^2(0,1)$. Therefore, we can integrate by parts in 
equation~\eqref{eq:psiRegPt} to obtain the differential equation
$
\psi''=(\psi')^2/\psi.
$

The function $t \mapsto R^t$ is the unique solution
of the latter differential equation satisfying the boundary conditions $\psi(0)=1$ 
and $\psi(1)=R$. Therefore, the solution $\psi$ of the
optimization problem \eqref{OptPr2} must be $\psi(t)=R^t$. But, there is
a function $h_\be\in Q^{1,2}$ such that $G(h_\be)\leq 0$ and
$J(h_\be)<J(\psi)$, in contradiction to the minimizing property of $\psi$.
In fact, a family of such functions is given by 
\begin{equation*}
h_\beta(t):=
\left\{
\begin{array}{lll}
1, &\mbox{ if } t\in[0,\be];\\
\frac{R^{2 \be}-1}{\be}(t-\be)+1,  &\mbox{ if } t\in(\be, 2\be]; \\
R^t, &\mbox{ if } t \in (2\be,1]
\end{array}
\right.
\end{equation*}
for $\be>0$ sufficiently small. 

The $C^2$ regularity of the solution $\psi$ of the optimization problem~\eqref{OptPr2} follows 
from~(i) and the special form of the Lagrangian
$$L(q,p)=(q^2-1)^2+(q-1)^2+\la\Big(\frac{p}{q}\Big)^2$$
associated with the functional $J+\la G$ for $\la>0$: it has positive second derivative with 
respect to $p$ on a neighborhood $U$ of the set
$\{(\psi(t),\psi'(t)):t\in[0,1]\}$ (see theorem~1.2.3, Sec.~1.2 in~\cite{JostCV}).
\end{proof}

\section*{Acknowledgments}
This project is supported in part by the NSF Grant DMS 0604331.

\end{document}